\documentclass[12pt,leqno,british]{amsart}
\usepackage{ifpdf}

\usepackage[T1]{fontenc}

\usepackage{eucal}
\usepackage{url}
\usepackage{stmaryrd}
\usepackage[pagebackref]{hyperref}
\usepackage{amsfonts}
\usepackage{amsthm}
\usepackage{amsmath}
\usepackage{amscd}
\usepackage{amssymb}

\def\CC {{\mathbb C}}     
\def\EE {{\mathbb E}}     
\def\NN {{\mathbb N}}     
\def\RR {{\mathbb R}}     
\def\ZZ {{\mathbb Z}}     

\def\ring#1{\ifmmode \mathaccent'027 #1\else \rm\accent'027 #1\fi}

\newcommand{\rI}{{\mathrm I}}

\newcommand{\ri}{{\mathrm i}}

\def\ol  {\overline}

\def\Lw  {\Longrightarrow}

\def\lw  {\longrightarrow}

\def\lo  {\longmapsto}

\def\wt  {\widetilde}

\def\mc {\mathcal}
\def\mk {\mathfrak}
\def\ms {\mathsf}

\def\be  {\begin{eqnarray}}
\def\ee  {\end{eqnarray}}
\def\ben {\begin{eqnarray*}}
\def\een {\end{eqnarray*}}

\def\bpr {\begin{proof}[Proof]}
\def\epr {\end{proof}}
\def\bsp {\begin{split}}
\def\esp {\end{split}}
\def\bprr {\begin{proof}[solution]}
\def\bpru {\begin{proof}[hint]}
\def\bpro {\begin{proof}[answer]}
\def\bcd {\begin{CD}}
\def\ecd {\end{CD}}

\newcommand{\abs}[1]{\left\vert#1\right\vert}
\newcommand{\scal}[1]{\left\langle#1\right\rangle}
\newcommand{\norm}[1]{\left\Vert#1\right\Vert}
\newcommand{\sco}[1]{\left(#1\right)}

\newcommand{\ksco}[1]{\left[#1\right]}

\newtheorem{theorem}{Theorem}[section]
\newtheorem{lemma}[theorem]{Lemma}
\newtheorem{prop}[theorem]{Proposition}
\newtheorem{coro}[theorem]{Corollary}
\newtheorem{remark}[theorem]{Remark}
\newtheorem{df}[theorem]{Definition}

\newtheorem{ex}[theorem]{Example}
\newtheorem*{ack}{Acknowledgements}

\begin{document}

\title[Homogeneous Hypercomplex Structures I ]\
{Homogeneous Hypercomplex Structures I - the compact Lie groups}

\author{George Dimitrov}
\address[Dimitrov]{University of Sofia "St. Kl. Ohridski"\\
Department of Mathematics and Informatics,\\ Blvd. James Bourchier
5,\\ 1164 Sofia, Bulgaria,}
\email{gkid@abv.bg}

\author{Vasil Tsanov}
\address[Tsanov]{University of Sofia "St. Kl. Ohridski"\\
department of Mathematics and Informatics,\\ Blvd. James Bourchier
5,\\ 1164 Sofia, Bulgaria,}
\email{tsanov@fmi.uni-sofia.bg}

\begin{abstract}
We introduce a remarkable subset "the stem" of the set of positive roots of a reduced root system. The stem
 determines several interesting decompositions of the corresponding reductive Lie algebra.
 It gives also a nice simple three dimensional subalgebra and a "Cayley transform". In the present paper we apply the above devices to give a complete classification of invariant hypercomplex structures on compact Lie groups.
\end{abstract}

\maketitle
\setcounter{tocdepth}{2}
\tableofcontents


\section{Introduction}

This paper is the first in a series of two, whose purpose is to give a description of compact hypercomplex homogeneous manifolds with a transitive action of a compact group.
The classification and proofs are entirely based on the structure theory of reductive Lie groups, it turns out that in the language of roots we get surprisingly clear answers to the natural questions.

  So we start with a complex manifold $(M,I)$ with a transitive compact group of biholomorphic automorphisms and look for another invariant complex structure $J$ on $M$, such that $IJ = -JI$ (we say shortly that $J$ {\bf matches} $I$). We call the complex structure $I$ {\bf admissible} if there exists a matching $J$.

 Our classification problem splits into two:

{\bf Problem  A}: In the class of compact complex homogeneous manifolds $(M, I)$, discern those which are admissible.

{\bf Problem B} Given an admissible complex structure $I$ on $M$, describe the class of all homogeneous hypercomplex structures on $M$ (up to  equivalence) of which $I$ is one of the complex structures.

In the present first paper (Section \ref{ste}) we introduce and discuss a remarkable invariant of reduced root systems, which we call the "stem" of $\Delta^+$.
The stem is a certain maximal strongly orthogonal set of roots, which is determined by $\Delta$ up to the action of the Weyl group (see Theorem \ref{ir1}).

Also in this paper we use the stem combinatorics to solve Problem A (see Theorem \ref{mai1} and corollaries) and Problem B (see Theorem \ref{mai2}) when our homogeneous space is the underlying manifold of a compact Lie group.

The idea to use a highest root to construct homogeneous "quaternionic" spaces goes back to Wolf \cite{Wolf65}. A wide class of examples of such structures was given by Spindel et al \cite{SpSeTrPr88} and Joyce \cite{Joyce92} where many ideas of the present paper may be traced in implicit form.

It is well known (\cite{Samelson53},\cite{Wang54}) that each compact even dimensional Lie group carries a homogeneous complex structure. A comprehensive description of the regular homogeneous complex structures on reductive Lie groups (not necessarily compact) in terms of structure theory may be found in Snow \cite{Snow86}.

For a compact Lie group our problems are easily reduced to
determining the hypercomplex structures on the Lie algebra, which
are integrable in the sense that the Nijenhuis tensor vanishes. We
show in particular that a compact simple Lie group ${\bf U}$ admits a left
invariant hypercomplex structure if and only if ${\bf U} =
SU(2k+1),\ k \geq 1$.

When our compact Lie algebra $\mk{u}$ is  "nearest to semisimple"\footnote{This notion is well defined in terms of the stem, it means that the rank of our compact Lie algebra is exactly twice the number of elements in the stem (see Corollary \ref{qqq3} and Theorem \ref{mai1}).}, a homogeneous complex structure may participate in at most one hypercomplex structure. More precisely: we use the stem to define a Cayley transform  of the Lie algebra, which determines completely the hypercomplex structure.

\subsection{Conventions and notations} \label{scn}

    Here we fix notations and recall well known facts, to be used throughout the paper.

 We shall denote by $\mk{u}$ a compact Lie algebra. Then the complexification  $\mk{u}^\CC = \mk{g} = \mk{g}_s \oplus \mk{c}$ is a reductive complex Lie algebra,
 whose semisimple ideal is $\mk{g}_s$, and the center is $\mk{c} \cong \CC^r$. We denote by $\tau$ conjugation of $\mk{g}$ w.r. to the real form $\mk{u}$,
 so $\tau$ is an antilinear involution of $\mk{g}$, such that $\mk{u} = \mk{g}^\tau = \mk{u}_s \oplus \mk{c}_u$.
 We denote by ${\bf U_s}$ and ${\bf G}_s$ the corresponding simply connected Lie groups,
 by ${\bf U} = {\bf U_s}\times{\bf C}_u, \ {\bf G} = {\bf G_s}\times {\bf C}$ - the corresponding reductive Lie groups (${\bf C}_u$ is a compact torus).

For $X,Y \in \mk{g}$, we denote by $\scal{X,Y}$ an ad-invariant symmetric bilinear form such that its restriction to the compact real form $\mk{u}$
is negative definite. We assume that $\scal{.,.}$ coincides with   the Killing form on the semisimple part $\mk{g}_s$.

Let $\mk{h}$ be a $\tau$-stable Cartan subalgebra of $\mk{g}$, then $\mk{h} = \mk{h}_s\oplus \mk{c}$, where $\mk{h}_s$ is a Cartan subalgebra of $\mk{g}_s$.
Let ${\bf H}$ be the corresponding Cartan subgroup of ${\bf G}$.
We denote by $\Delta$ the root system of $\mk{g}_s$ w.r. to $\mk{h}_s$. For
$\alpha \in \Delta$ we denote by $h_{\alpha}$ the element of
$\mk{h}$  determined by $\scal{H,h_{\alpha}} = \alpha(H)$ for all
$H \in \mk{h}$, and we denote
$$
H_\alpha = \frac{2}{\scal{\alpha,\alpha}}h_\alpha,\qquad \mk{g}(\alpha) = \{X \in \mk{g}\vert \ adH(X) = \alpha(H)X,\ H \in \mk{h} \}.
 $$
Further for $\alpha,\beta \in \Delta $ we denote
$$C(\beta,\alpha) = \frac{2\scal{\beta,\alpha}}{\scal{\alpha,\alpha}},\qquad s_\alpha(\beta) = \beta - C(\beta,\alpha)\alpha.
$$
The map $\beta \lo s_\alpha(\beta)$ is the reflection along $\alpha$ (see e.g. \cite{Helgason78}, ch III).

By ${\bf Aut}(\Delta)$ we denote the group of all the elements in $GL(\mk{h}_{\RR}^*)$, which leave the set $\Delta \subset \mk{h}_s^*$
 invariant and the center $\mk{c}$ pointwise fixed.

The {\bf Weyl group} ${\bf W} = {\bf W}(\Delta)$ is the (normal)
subgroup of ${\bf Aut}(\Delta)$, which is generated by all
reflections $s_\alpha, \alpha\in \Delta$. The Weyl group acts
simply transitively on the set of all bases $\Pi$ of $\Delta$.
 For a fixed basis $\Pi$ of the root system $\Delta$ we denote
${\bf Aut}_{\Pi}(\Delta) = \{\phi \in {\bf Aut}(\Delta)\vert\ \phi(\Pi) = \Pi\}$.

The adjoint action of the Weyl group ${\bf W}$ on $\mk{h}$ is defined for $s\in {\bf W} $ by
$\alpha(s(H)) = s^{-1}(\alpha)(H),$ $ H \in \mk{h}$. For any $\gamma \in \Delta$ we have
$s_\gamma(H) =  H - \gamma(H) H_\gamma, \quad H \in \mk{h}$.
 The normalizer ${\bf N}\subset {\bf G}$ of the Cartan subalgebra $\mk{h}$ is $
{\bf N} = {\bf N}(\mk{h}) = \{g \in {\bf G} \vert\ Ad g(\mk{h}) = \mk{h}\},\quad {\bf N}_u = {\bf N}\cap{\bf U}$.
The following   exact sequence is a fundamental fact of structure theory
 \begin{equation}\label{elm1}
\bcd
1 @>>>{\bf H}_u @>\iota>> {\bf N}_u @>>> {\bf W} @>>> 1.
\ecd
\end{equation}
\subsubsection{Weyl - Chevalley basis}
We (may) choose elements $E_\alpha\in \mk{g}(\alpha)$, so that the structural constants are integer, i.e., for $\alpha, \beta, \alpha + \beta \in \Delta$:
\begin{gather}
[E_\alpha, E_{- \alpha}] = H_\alpha, \quad [E_{\alpha}, E_{\beta}] = N_{\alpha,\beta}E_{\alpha + \beta},\nonumber\\
\label{hel1}\\
N_{\alpha,\beta} = - N_{-\alpha,-\beta},\quad \abs{N_{\alpha,\beta}} = 1 - p,\nonumber
\end{gather}
where $\beta + n \alpha,\quad p \leq n \leq q$ is the $\alpha$-series of $\beta$ (see e.g. \cite{Helgason78}, ch III).

It is  convenient to extend \eqref{hel1} and {\bf define the symbol} $ N_{\alpha, \beta}$ for any couple of functionals $\alpha, \beta \in \mk{h}^*$ by
\be \label{structural constants}
N_{\alpha,\beta}= 0, \hbox{ if } \alpha \not \in \Delta \mbox{ or } \beta \not \in \Delta \mbox{ or } \alpha + \beta \not \in \Delta.
\ee

In the above basis for $\mk{g}_s$, the {\bf contragredience} involution  $\theta \in {\bf Aut}(\mk{g})$ is the complex linear map determined by
\begin{gather}\label{star2}
\theta(E_\alpha) = -E_{-\alpha},\ \theta(H) = - H,\quad  \alpha \in \Delta,\ H\in \mk{h}.
\end{gather}
The conjugation $\tau$ is the {\bf antilinear involution} given  by
\begin{gather*}
\tau(E_{\alpha}) = - E_{-\alpha}\quad \tau(H_\alpha) = - H_\alpha,\quad \tau((z_1,\dots,z_r)) = (-\ol{z}_1,\dots,-\ol{z}_r).
\end{gather*}
where $(z_1,\dots,z_r)\in\mk{c}$.

We have $\mk{u} = \mk{g}^\tau = \{X\in \mk{g}\vert\ \tau(X) = X \}$. As $\mk{h}$ is $\tau$-invariant:
\begin{gather} \label{re1}
\alpha(\tau(H)) = -\ol{\alpha(H)},\ \alpha \in \Delta,  H \in \mk{h}.
\end{gather}
We now fix a {\bf basis} $\Pi = \{\alpha_1,\dots,\alpha_l\}$ of $\Delta$, which gives us a {\bf system of positive roots} $\Delta^+$. We denote
\begin{gather*}
     \mk{n}^\pm = \bigoplus_{\alpha \in \Delta^\pm} \mk{g}(\alpha),\quad \mk{g} = \mk{h} \oplus \mk{n}^+ \oplus \mk{n}^-,\quad \mk{b}^\pm = \mk{h} \oplus \mk{n}^\pm.
\end{gather*}
The {\bf Borel subalgebra} $\mk{b}^+$ is a maximal solvable subalgebra of $\mk{g}$.

For any $\gamma \in \Delta$ we denote
\begin{gather*}
sl_\gamma(2) = span_\CC\{E_\gamma,E_{-\gamma}, H_\gamma\}\subset \mk{g},\quad su_\gamma(2) = \mk{u}\cap sl_\gamma(2).
\end{gather*}
\begin{df}
A subalgebra $\mk{a} \subset \mk{g}$  is called $\mk{h}$-{\bf regular} if its normalizer
$\mk{n}(\mk{a})$ contains a Cartan subalgebra $\mk{h}$ of $\mk{g}$. A subalgebra $\mk{a}$ is called {\bf regular} if it is $\mk{h}$-regular for some
Cartan subalgebra $\mk{h}$.
\end{df}
It is well known that if
$\mk{a}$ is an $\mk{h}$-regular subalgebra of $\mk{g}$, then we have a decomposition:
\begin{gather}\label{li3}
 \mk{a} = (\mk{h}\cap\mk{a}) \oplus \bigoplus_{\alpha \in \Theta}
\mk{g}({\alpha}),\mbox{ where }\quad \Theta = \{ \alpha \in \Delta \vert\ \mk{g}({\alpha}) \subset\mk{a}\}.
\end{gather}

\subsection{Complex structures on a compact Lie group } \label{css}

Any left invariant almost complex structure on the manifold ${\bf U}$, determines
(and is determined by) a complex structure $I: \mk{u} \rightarrow \mk{u}$. The obvious condition for the existence of a complex structure on $\mk{u}$ is  even dimension, and this is the same as even rank.

 It is clear that an invariant (hyper)comp;ex structure on ${\bf U}$ determines and is determined by an invariant (hyper)complex structure on the universal covering group $\wt{\bf U}  \cong {\bf U}_s \times \RR^r$. We have ${\bf U} = \wt{\bf U}/\Lambda$, where $\Lambda$ is some central lattice in $\wt{\bf U}$. It is well known that equivalent complex structures on $\wt{\bf U}$ may project to unequivalent complex structures on ${\bf U}$. In this paper we concentrate rather on classifying up to equivalence the invariant hypercomplex structures  on $\wt{\bf U}$, which is done in terms of data on the Lie algebra $\mk{u}$. The dependence on $\Lambda$ is well understood in the literature.

Let $I$ be any complex structure on $\mk{u}$. We extend $I$ to
$\mk{g}$ (and go on to denote the extension by $I$) setting $I(\ri
X) = \ri  IX$. Thus on $\mk{g}$ we have $I\circ \tau = \tau \circ
I$.
\begin{df}\label{cs2}
Let  $I$ be a complex structure on $\mk{u}$. We denote
 \begin{gather*}
 \mk{m}^+_I = \{X\in \mk{g}\vert IX =  \ri X\} = \{X  - \ri IX\vert X \in \mk{u}\};\\
 \mk{m}^-_I = \{X\in \mk{g}\vert IX =  - \ri X\} = \{X  + \ri IX\vert X \in \mk{u}\} = \tau( \mk{m}^+_I).
 \end{gather*}
 \end{df}
In other words: $\mk{m}^+_{I},\  \mk{m}^-_I$ are respectively the $(1,0)$ and $(0,1)$ components (w.r. to the left invariant almost complex structure $I$)
of the complexified tangent space to ${\bf U}$ at the unit element. It is also well known (and obvious) that
\begin{gather}\label{oo8}
\mk{g} = \mk{m}^+_I \oplus \mk{m}^-_I.
\end{gather}
If $I$ is a complex structure on $\mk{u}$ we define its {\bf Nijenhuis tensor}:
\begin{gather}\label{cs3}
N_I(X,Y) = [IX,IY] - I[I X, Y] - I[X,IY] - [X,Y],\quad X,Y \in \mk{u}.
\end{gather}
It is often convenient to "complexify" the Nijenhuis tensor by allowing $X, Y$ in the above formula to vary in $\mk{g} = \mk{u}^\CC$.
We denote the complexified Nijenhuis tensor by the same letter.

 The following proposition is well known (see e.g. Snow \cite{Snow86})
 \begin{prop}\label{li1}
The left invariant almost complex structure induced by $I$ on ${\bf U}$ is a complex structure if and only if,
any one of the following conditions is satisfied:

a) $\mk{m}^+_I$ is a subalgebra of $\mk{g}$; \quad b)  $N_I \equiv 0$.
\end{prop}
In this paper, a complex structure on the compact Lie algebra $\mk{u}$ will be
called  {\bf integrable}, if it satisfies the conditions from Proposition \ref{li1}.
\begin{df}
Two complex structures $I, I'$ on $\mk{u}$ will be called {\bf equivalent} if there exists an automorphism $\xi$ of $\mk{u}$ such that $\xi\circ I = I'\circ \xi$.
\end{df}

\begin{df}
We shall say that a complex structure $I$ on a Lie algebra $\mk{u}$ is {\bf regular} if $\mk{m}^+_I$ is a regular subalgebra w.r. to some $\tau$-stable
Cartan subalgebra $\mk{h}$ of $\mk{u}^\CC$.
\end{df}

Since ${\bf U}$ is compact, we may assume that $I$ is  a regular complex structure (see Snow, \cite{Snow86}). Throughout
the paper $\mk{h}$ will denote a $\tau$-stable Cartan subalgebra in the normalizer of
$\mk{m}^+_I$.

 Let $\Delta \subset (\mk{h})^{*} $ be the root system of $\mk{g}$ w.r. to $\mk{h}$. We have

\begin{prop} \label{lem1} An integrable complex structure $I$ on $\mk{u}$
determines a system of positive roots $\Delta^+$, and
a subspace $\mk{h}^+ = \mk{m}^+_I\cap\mk{h} \subset \mk{h}$, such that
\begin{gather*}
\mk{m}^+_I = \mk{h}^+ \oplus \mk{n}^+, \quad \mk{h} = \mk{h}^+
\oplus \mk{h}^- ,\mbox{ where }   \mk{h}^- = \tau(\mk{h}^+) = \mk{m}^-_I\cap\mk{h} .
\end{gather*}
In particular $dim(\mk{h}^+) = \frac{1}{2} dim(\mk{h})$.
\end{prop}
\begin{proof}
From regularity of $I$ we have the decomposition \eqref{li3}.

The Cartan subalgebra $\mk{h}$ is $\tau$-invariant, whence
$\mk{g} = \mk{m}^+_I \oplus \tau(\mk{m}^+_I)$ implies
$\mk{h}^+\oplus \mk{h}^- = \mk{h}$, whence the last statement of the lemma.

If $\alpha  \in \Theta$ and $-\alpha \in \Theta$,  then $H_\alpha = [E_{\alpha}, E_{-\alpha}] \in \mk{m}^+_I$, but $\tau(H_\alpha) = -H_\alpha$, whence $\mk{m}^+_I \cap \tau(\mk{m}^+_I) \neq \emptyset$,
 which contradicts Proposition \ref{li1}.
Because $dim(\mk{m}^+_I) = \frac{1}{2}dim(\mk{g})$,
we conclude that  $\Theta$ contains exactly one of the roots in each couple
$\{\alpha,-\alpha\} \subset \Delta$.   But $\mk{m}_I^+$ is also a subalgebra, so  $\Theta = \Delta^+$ for some basis of $\Delta$ (see e.g. \cite{Bourbaki68}, Ch VIII, Sect 3, Prop. 7). The lemma is proved.
\end{proof}

\begin{remark}\label{md12}
If $I$ is a regular complex structure on a noncompact reductive Lie algebra $\mk{g}_0$, then the subalgebra $\mk{m}^+_I$ may have a nontrivial Levy component (see e.g. Snow \cite{Snow86}) .
\end{remark}

Throughout this paper, given an integrable complex structure $I$ on $\mk{u} = \mk{g}^\tau$ we shall denote the corresponding $\tau$-invariant Cartan  subalgebra $\mk{h} = \mk{h}_I$, the subspace $\mk{h}^+ = \mk{h}^+_I$ with dimension $m = dim(\mk{h}^+) = \frac{1}{2} dim(\mk{h})$,
 the Borel subalgebra $\mk{b}^+ = \mk{b}_I^+ = \mk{h}_I\oplus \mk{n}_I^+$ etc.
When (we believe that) no confusion may arise, we shall omit the subscript $I$. When we have to refer to this connection between $I$ and the structural data,
we shall say briefly that {\bf $I$ is a $\mk{b}^+$ complex structure}.
In other words, a complex structure $I$ on $\mk{u}$, will be called a {\bf $\mk{b}^+$-complex structure} iff  $\mk{b}^+$ is the normalizer of $\mk{m}^+_I$.

It is well known that ${\bf Ad}\mk{u}$ acts transitively on the set of all Borel subalgebras of $\mk{g}$, thus if we fix a Borel subalgebra $\mk{b}^+$, then any integrable complex structure on $\mk{u}$ is equivalent to  a $\mk{b}^+$-complex structure.

\begin{remark}\label{md1}
It is well known that a compact group ${\bf U}$ may have a left invariant complex structure $I$ in such a way, that the simple factors are not complex submanifolds. Perhaps the best known semisimple example is a Calabi-Eckman invariant complex structure on $SU(2)\times SU(2)$. For examples with even dimensional factors s.
\end{remark}

\subsection{Left invariant almost hypercomplex structures}\label{hss}

\begin{df}
A {\bf left invariant almost hypercomplex structure} on ${\bf U}$ is a couple of complex structures $I, J:\mk{u} \lw \mk{u}$, which anti-commute i.e. $I\circ J = - J\circ I$. An almost hypercomplex structure will be called a {\bf hypercomplex structure} if both $I,J$ are integrable.

 Two hypercomplex structures $(I,J), (I',J')$ on $\mk{u}$ will be called {\bf equivalent} if there exists an automorphism $\xi$ of $\mk{u}$ such
 that $\xi\circ I = I'\circ \xi,\ \xi\circ J = J'\circ \xi $.

  We use the same letters to denote the complexifications of the operators  $I, J$, so we have two linear maps $I, J : \mk{g} \lw \mk{g}$, such that
\begin{gather}\label{hks1}
IJ = -JI,\quad I^2 = J^2 = -1,\quad \tau\circ I = I\circ\tau,\quad \tau\circ J = J\circ\tau.
\end{gather}
\end{df}
 First we show
 \begin{lemma}\label{oo1}
Let $I, J$ be complex structures on $\mk{u}$. Then $I \circ J = - J \circ I$ if and only if  $ I(\mk{m}^+_J) = \mk{m}^-_J $.
\end{lemma}
\begin{proof}  If $I \circ J =- J \circ I$, then for $X \in \mk{m}^+_J$ we have $JIX = - IJX= - \ri
IX$.  \\
 If $I(\mk{m}^+_J) = \mk{m}^-_J$, then for $X \in \mk{m}^+_J$ we have $JIX = - \ri IX$ and $IJX= \ri
IX$, hence $(I \circ J)_{\vert \mk{m}^+_J}= - (J \circ I)_{\vert
\mk{m}^+_J}$. Since $I$ and $J$ commute with $\tau$ and $\mk{m}^+_J\oplus
\mk{m}^-_J = \mk{g}$, we have $I \circ J = - J\circ I$.
\end{proof}
\begin{df}
Let $\mk{u}$ be a compact Lie algebra. Let $I$ be a $\mk{b}^+$ complex structure as described in Subsection \ref{css}.
We shall say that a complex structure $J$ on $\mk{u}$ {\bf matches $I$} if $J$ is integrable and $IJ = - JI$.
We call $I$ {\bf admissible} if there exists some $J$, which matches $I$.
\end{df}

Now we introduce more notation, which will be used throughout the paper. We are interested in hypercomplex structures, so from this moment {\bf we assume that we have fixed a $\mk{b}^+$ complex structure $I$ on $\mk{u}$} and use freely the notations from subsection \ref{css} and Proposition \ref{lem1}. Further we assume that $J$ is a complex structure on $\mk{u}$, such that $JI = - IJ$.
\begin{df}\label{bo2}
We fix a basis $U_1,\dots,U_m$ of $\mk{h}^+$, then we define $V_k = \tau(U_k)\in \mk{h}^-$ so that we have bases
\begin{gather*}
\{E_\alpha\vert \alpha \in \Delta^+\}\cup \{U_1,\dots,U_m\} \mbox{  of  } \mk{m}^+_I;\\ \{E_{-\alpha}\vert \alpha \in \Delta^+\}\cup\{V_1,\dots,V_m\} \mbox{ of } \mk{m}^-_I.
\end{gather*}
For  $\alpha \in \Delta^+,\ q=1,\dots, m$ we decompose the elements $JE_\alpha, JU_q$ as follows
\begin{gather}
JE_\alpha  =
\sum_{\beta \in \Delta^+} a_{\beta,\alpha} E_{-\beta} +  \sum_{t=1}^m \xi_{t,\alpha} V_t; \nonumber \\
\label{bo6}\\
JU_q  = \sum_{\beta \in \Delta^+} \eta_{\beta,q} E_{-\beta} + \sum_{t=1}^m  b_{t, q}V_t.\nonumber
\end{gather}
We introduce matrices with coefficients $a_{\alpha,\beta},\ b_{t,q},\ \xi_{t,\alpha},\ \eta_{\alpha,q}$ respectively:
$${\bf a} \in {\mc M}(n \times n);\quad {\bf b} \in {\mc M}(m \times m);\quad \xi \in {\mc M}(m
\times n);\quad \eta \in {\mc M}(n\times m).$$
\end{df}
\begin{prop}\label{lemma about J} Let $J$ be a complex structure on $\mk{u}$, such that
$J \circ I = - I \circ J$. In the bases of Definition \ref{bo2} the linear operator $J$ has the matrix
\begin{gather}\label{bo8}
\ms{J} =
\begin{bmatrix}
0 & 0 & \ol{\bf a}&-\ol{\eta}\\
0 & 0 & -\ol{\xi}&\ol{\bf b}\\
  {\bf a}& \eta &0&0\\
  \xi & {\bf b}&0&0
\end{bmatrix},\qquad
\begin{matrix}
\ol{\eta} \xi  - \ol{{\bf a}} {\bf a}  = \rI_{n},\ &  \ol{{\bf b}}\xi - \ol{\xi} {\bf a}  = 0,\\
\ol{\bf a} \eta -  \ol{\eta }{\bf b}  = 0, \ &  \ol{\xi}\eta - \ol{{\bf b}} {\bf b}= \rI_m.
\end{matrix}
\end{gather}
Conversely, for any choice of ${\bf a}, {\bf b}, \xi, \eta$ as in \eqref{bo8}, the operator given by the matrix $\ms{J}$
 commutes with $\tau$ and defines a complex structure $J$ on $\mk{u}$, such that $J \circ I = - I \circ J $.
\end{prop}
\begin{proof}
Using $J\circ\tau = \tau\circ J$ and \eqref{bo6} we compute
\begin{gather*}
 JE_{-\alpha} = -\tau(JE_\alpha) =  \sum_{\beta \in \Delta^+}
\ol{a_{\beta,\alpha}} E_\beta -  \sum_{t=1}^m \ol{\xi_{t,\alpha}} U_t, \\
JV_q = \tau (JU_q) = -\sum_{\beta \in \Delta^+} \ol{\eta_{\beta,q}} E_\beta  +  \sum_{t=1}^m \ol{ b_{t,q}}U_t.
\end{gather*}
The equalities in \eqref{bo8} mean the same as $\ms{J}^2 = - \rI$.
\end{proof}
Obviously, many invariant almost hypercomplex structures on ${\bf U}$ exist iff
$dim(\mk{u})$ is divisible by 4.

\section{Stems}\label{ste}
Throughout this section $\Delta$ is a reduced root system, $\Pi$ is a basis of $\Delta$ and $\Delta^+$ is the corresponding subset of positive roots.

\begin{df}\label{aa1}
For any $\gamma \in \Delta^+$ we denote
$$\Phi_\gamma^+ \doteq \{\alpha \in \Delta^+\vert\  \gamma - \alpha \in \Delta^+\}.
$$
A subset $\Gamma \subset \Delta^+$ will be called a {\bf stem of} $\Delta^+$ iff
\begin{gather} \label{nos1}
\Delta^+ = \Gamma \cup \bigcup_{\gamma \in \Gamma}\Phi_\gamma^+, \qquad \mbox{disjoint union}.
\end{gather}
If $\Gamma$ is a stem of $\Delta^+$ and $\gamma \in \Gamma$, we shall call $\Phi_\gamma^+$ the {\bf  branch} at $\gamma$.
\end{df}
In the present section we shall prove existence and uniqueness of a stem for a reduced root system $\Delta$ with a fixed basis $\Pi$ (hence fixed $\Delta^+$).
We also derive the properties of stems needed for applications to the existence and properties of hypercomplex structures.
 Next we give a list of notations related to a stem.

\begin{df}\label{wj1}
Let $\Gamma$ be a stem of $\Delta^+$. We denote
\begin{gather}
\Phi_\gamma^- = -\Phi_\gamma^+,\quad \Phi_\gamma = \Phi_\gamma^+\cup \Phi_\gamma^- ,\nonumber \\[-2mm] \label{ste1} \\[-2mm]
\Phi^+ = \bigcup_{\gamma\in \Gamma} \Phi_\gamma^+,\quad \Phi^- = -\Phi^+, \quad  \Phi= \Phi^+\cup \Phi^-. \nonumber
\end{gather}
\end{df}
So we have a disjoint union $\Delta^+ = \Gamma \cup \Phi^+.$

\subsection{Existence and uniqueness of the stem}\label{eu}
\begin{df}\label{lm1}

Let $\Delta = \Delta_1\cup\dots\cup\Delta_k$ be the decomposition of
$\Delta$ into mutually orthogonal, irreducible root subsystems.

We shall say that $\gamma\in \Delta$ is a {\bf long root} if
 $\norm{\gamma} \geq \norm{\alpha}$ for each $\alpha \in
\Delta_j$, where $\gamma \in \Delta_j$.

We shall say that $\gamma\in \Delta$ is a {\bf maximal root} if
 $\gamma$ is the highest root in some $\Delta_j$.

We shall say that two roots $\alpha,\beta \in \Delta$ are {\bf strongly orthogonal} if
 $\alpha \pm \beta \not\in \Delta $.
\end{df}

\begin{prop}\label{gor2}Let $\gamma\in \Delta^+$ be a long root, $\alpha\in \Delta,\quad \alpha \neq \pm\gamma$.
Let the $\gamma$-series of $\alpha$ be $\{\alpha +n\gamma,\ p\leq n \leq q\}$ (see e.g. \cite{Helgason78}). Then

a) $\abs{C(\alpha,\gamma)} \leq 1$.

b) The number of distinct roots in the $\gamma$-series of $\alpha$ is at most 2. More precisely
\begin{gather}\label{len1}
q-p+1 = \begin{cases}
1 &\mbox{ if and only if  } C(\alpha,\gamma) = 0,\\
2 &\mbox{ if and only if  } |C(\alpha,\gamma)| = 1.
\end{cases}
\end{gather}
\indent c) $\scal{\alpha,\gamma} = 0$ if and only if  $\gamma$ is strongly orthogonal to $\alpha$.
\end{prop}
\begin{proof}
a) If $\gamma$ and $\alpha$ belong to different irreducible components of $\Delta$, then $C(\alpha,\gamma) = 0$.
If  they belong to the same component, the claim follows from the Schwartz inequality.

b) We have $p + q = - C(\alpha,\gamma)$ (see e.g.
\cite{Bourbaki68}). Thus the length of the $\gamma$- series of
$\alpha$ is $1 - C(\alpha +p\gamma,\gamma) \leq 2$, the last
inequality obviously follows from a). Thus either $p=0$ or $q=0$.
But then obviously both vanish iff $C(\alpha,\gamma)= 0$.

c) trivially follows from b).
\end{proof}

\begin{prop}\label{gor6} Let $\gamma$ be a maximal root. Then

a) $\gamma$ is a long root;

b) If $\alpha \in \Delta^+$ and $\alpha \neq \gamma$, then $ 0 \leq C(\alpha,\gamma)\leq 1$.

c) Let $\alpha \in \Delta$ and $\alpha \neq \gamma$. Then $\alpha \in
\Phi_\gamma^+$ iff $C(\alpha,\gamma) = 1$.

\end{prop}
\begin{proof}
Claims a), b) are proved e.g. in (\cite{Bourbaki68}, Ch.VI, Sect.1.8).

c) If $C(\alpha,\gamma) = 1$, then $\alpha -\gamma = s_\gamma(\alpha) \in \Delta$. By maximality of $\gamma$ we have $\alpha - \gamma \in \Delta^-$.
Whence $\gamma - \alpha \in \Delta^+$.
\end{proof}

\begin{prop} \label{nsr6} Let $\gamma$ be a maximal root. Let $\alpha \in \Phi_\gamma^+,\ \nu \in \Delta,\ \nu\neq \pm\gamma$. Then

a) If $\nu \in \Phi_\gamma^+$ and $\alpha + \nu \in \Delta$, then $\alpha + \nu = \gamma$.

b)  If  $\nu \not\in \Phi_\gamma$ and $\alpha + \nu \in \Delta$, then $\alpha + \nu \in \Phi_\gamma^+$.
 \end{prop}
\begin{proof} If $\alpha,\nu \in\Phi_\gamma^+$, then by Proposition \ref{gor6}, c), we have
$$
C(\alpha + \nu, \gamma) = C(\alpha, \gamma) + C(\nu, \gamma) = 2.
$$
By Proposition \ref{gor2}, a) we have $\alpha + \nu = \pm\gamma$, but $\alpha + \nu \in \Delta^+$, so a) is proved.

Now we prove b).  By Proposition \ref{gor6} c), $\nu \not \in
\Phi_\gamma$ implies $\scal{\nu,\gamma} = 0$ whence
$$
C(\alpha + \nu,\gamma) = C(\alpha,\gamma) + C(\nu,\gamma) = 1.
$$
Now, if $\alpha + \nu \in \Delta$, by Proposition \ref{gor6}, c), we have
$\alpha + \nu \in \Phi_\gamma^+$.
\end{proof}

\begin{prop}\label{yyy8}
Let $\gamma \in \Delta^+$ be a maximal root and let $\Pi$ be our fixed basis of $\Delta$. The set $\Phi_\gamma^+\cap\Pi$ has at most two elements. Also
$$\Phi_\gamma^+ = \emptyset \iff \gamma \in \Pi \iff  \Phi_\gamma^+\cap\Pi = \emptyset.$$
\end{prop}
\begin{proof} Without loss of generality we may assume that $\Delta$ is irreducible and $\gamma$ is the highest root.

The first equivalence claimed is just the definition of a simple root.
Any root $\gamma \in \Delta^+$ has a representation $\gamma = \beta_1+ \beta_2 + \dots+\beta_k$ where all summands are simple roots and
each partial sum is a root (see e.g. \cite{Helgason78}, Ch.X, Lemma 3.10 ).
The last root in the sequence belongs to $\Phi_\gamma^+\cap\Pi$. Whence the second equivalence follows.

How we treat the case when $\Phi_\gamma^+ \neq \emptyset$. We decompose
\begin{gather}\label{mar1}
\gamma = \sum_{\alpha\in\Pi}n_\alpha(\gamma)\alpha,
\end{gather}
where $n_\alpha(\gamma) \in \NN$ for each $\alpha\in\Pi$.
By Proposition \ref{gor6}, c) we have
$$
2 = C(\gamma,\gamma) = \sum_{\alpha\in\Pi}n_\alpha C(\alpha,\gamma) = \sum_{\alpha \in\Phi_\gamma^+\cap\Pi}n_\alpha.
$$
So if $\xi \in\Phi_\gamma^+\cap\Pi$, then either $n_\xi = 2$ and $\Phi_\gamma^+\cap\Pi =\{\xi\}$, or $n_\xi = 1$, and there is exactly one element $\eta \in \Phi_\gamma^+\cap\Pi$ with $\eta \neq \xi$.
\end{proof}

We shall need some properties of the orthogonal complement of a maximal root.
\begin{lemma}\label{nsr3}
Let $\gamma$  be a maximal root in $\Delta^+$. If $\widetilde{\Delta}=\Delta \setminus (\Phi_\gamma \cup \{\gamma,-\gamma\})$, then

 a) $\widetilde{\Delta}$ is a reduced root system,  $\wt{\Pi} = \Pi\cap\wt{\Delta} $ is a basis of $\widetilde{\Delta}$, and
 the corresponding subset of positive roots is $\widetilde{\Delta}^{+} = \Delta^+ \setminus (\Phi_\gamma^+ \cup \{\gamma\}) = \widetilde{\Delta}\cap\Delta^+ $.

 b) If $\alpha, \beta \in \widetilde{\Delta}$ and $\alpha + \beta \in \Delta$, then $\alpha + \beta \in \widetilde{\Delta}$.

 c) For any $\alpha \in \widetilde{\Delta}^+$ we have
$\quad \Phi_\alpha^+ =  \{\beta \in \Delta^+\vert\ \alpha - \beta \in \Delta^+ \} = \{\beta \in \widetilde{\Delta}^+\vert\ \alpha - \beta \in \widetilde{\Delta}^+ \}.
$
\end{lemma}
\bpr
a) and b) follow from the fact (see Proposition \ref{gor6}) that $\widetilde{\Delta} = \{\alpha \in \Delta; \scal{\alpha,\gamma} = 0\}$.

c) Let  $\alpha - \beta \in \Delta^+$. Let us assume that
$\alpha - \beta \not \in \widetilde{\Delta}^+ $. Therefore by a) we have $\alpha -
\beta \in \Phi_{\gamma}^+ \cup \{\gamma\}$. But,
$\alpha-\beta \neq \gamma$, since $\gamma$ is a maximal
root. So $\alpha - \beta \in \Phi_{\gamma}^+ $.

Recalling that $\beta \in \Delta^+$, we see that $\beta \in
\Phi_{\gamma}^+$ or $\beta \in \widetilde{\Delta}^{+}$ (now,
obviously, $\beta = \gamma$ is impossible).

If $\beta \in \Phi_{\gamma}^+$ then by Proposition \ref{nsr6} a)
we obtain $\alpha = (\alpha - \beta)+\beta = \gamma$, which
contradicts $\alpha \in \widetilde{\Delta}^+$.
If $\beta \in \widetilde{\Delta}^{+}$ then by Proposition
\ref{nsr6} b) we obtain $\alpha = (\alpha - \beta)+\beta \in
\Phi_{\gamma}^+$, which contradicts $\alpha \in
\widetilde{\Delta}^+$.

Thus, we proved $\alpha - \beta  \in \widetilde{\Delta}^+.$

Now, using Proposition \ref{nsr6} again, one can easily show that
$\beta \in \widetilde{\Delta}^{+}$. The lemma is proved.
\epr

The construction of a stem is contained in the following
\begin{prop}\label{eu1}
There exists a sequence  $\Delta = \Delta_1 \supset \Delta_2 \supset \dots \supset \Delta_d$ of closed root subsystems\footnote{One says that a subsystem $\Theta \subset \Delta$ is {\bf closed} iff $\alpha,\beta \in \Theta$ and $\alpha + \beta \in \Delta$ imply $\alpha + \beta \in \Theta$.} with bases $\Pi_k = \Pi\cap\Delta_k$,
 corresponding sets of positive roots $\Delta_k^+  = \Delta^+ \cap\Delta_k$   and maximal roots $\gamma_k \mbox{ of } \Delta_k^+,\quad k = 1,\dots,d$,
  such that we have disjoint unions:
\begin{gather} \nonumber
\Delta_1^+ = \Phi_{\gamma_1}^+ \cup \{\gamma_1\} \cup \Delta_2^+ ,\quad \dots, \quad \Delta_{d-1}^+ = \Phi_{\gamma_{d-1}}^+ \cup \{\gamma_{d-1}\} \cup \Delta_d^+, \\[-2mm] \label{eu2} \\[-2mm] \nonumber \Delta_d^+=\Phi_{\gamma_d}^+ \cup \{\gamma_d\}.
\end{gather}
The set $\Gamma = \{\gamma_1,\dots,\gamma_d\}$ is a stem of $\Delta^+$.
\end{prop}
\begin{remark}
The construction goes by induction, taking at each step a maximal root $\gamma_k\in \Delta_k^+$ and defining  $\Delta_{k+1} = \gamma_k^\perp = \{\alpha\in \Delta_k\vert\ \scal{\alpha,\gamma_k} = 0 \}$. The point is to prove that for each $k = 1,\dots,d$ we have
\begin{gather}\label{sol}
\Phi_{\gamma_k}^+ = \{\alpha \in \Delta^+ \vert\ \gamma_k - \alpha \in \Delta^+\} = \{\alpha \in \Delta_k^+ \vert\ \gamma_k-\alpha \in \Delta_k^+ \}.
\end{gather}
The induction step is based on Lemma \ref{nsr3}, c).
\end{remark}

\bpr
Let $\gamma_1 \in  \Delta^+$ be a maximal root.  We put $\Delta_1 = \Delta,\ \Delta_1^+ = \Delta^+$ and define:
\begin{gather*}
\Delta_2 = \{\alpha\in \Delta_1\vert\ \scal{\alpha,\gamma_1} = 0 \},\quad \Pi_2 =  \Pi \cap\Delta_2,\quad \Delta_2^+ =  \Delta^+ \cap\Delta_2.
\end{gather*}
By Proposition \ref{gor6},c), we have $\Delta_1^+ = \Phi_{\gamma_1}^+ \cup \{\gamma_1\} \cup \Delta_2^+$. If $\Delta_2 = \emptyset$,
then obviously $\Gamma = \{\gamma_1\}$ is a stem, and our lemma is proved.

If $\Delta_2 \neq \emptyset$, we go on by induction. Assume that $k > 1$ and we have defined a sequence of root systems
$\Delta_1 \supset \Delta_2 \supset \dots \supset \Delta_k \neq \emptyset$   and maximal roots
$\gamma_1 \mbox{ of } \Delta_1^+,  \dots, \gamma_{k} \mbox{ of } \Delta_{k}^+$, such that for $i=1,\dots,k-1$ we have
\begin{gather} \nonumber
\Delta_{i+1} = \{\alpha \in \Delta_i \vert \ \scal{\gamma_i,\alpha} = 0\},\quad \Pi_{i+1} =  \Pi \cap\Delta_{i+1} , \nonumber \\ \label{induction ass 2} \Delta_i^+ =  \Delta^+ \cap\Delta_i = \Phi_{\gamma_i}^+ \cup \{\gamma_i\} \cup \Delta_{i+1}^{+} ;\\ \nonumber
  \alpha \in \Delta_{i+1}^+  \Lw \{\beta \in \Delta_{i+1}^+ \vert\ \alpha-\beta \in \Delta_{i+1}^+ \} = \Phi_{\alpha}^+.
\end{gather}
The above conditions are obviously valid for $k=2$ (see Proposition \ref{nsr6} and Lemma \ref{nsr3}).

To make the induction step we choose a maximal root $\gamma_k\in\Delta_k^+$ to get the sequence
$\gamma_1,\dots, \gamma_{k-1}, \gamma_k$.  By \eqref{induction ass 2} with $i=k-1$, we see that
 \be \label{for Delta_{k+1}0}
  \Phi_{\gamma_k}^+ = \{\beta \in \Delta_{k}^+ \vert\ \gamma_k - \beta \in \Delta_{k}^+ \} \subset \Delta_k^+.
  \ee
Now we define
$$
\Delta_{k+1} = \{\alpha\in \Delta_k\vert\ \scal{\alpha,\gamma_k} = 0 \},\quad \Pi_{k+1} =  \Pi \cap\Delta_{k+1}.
$$
Obviously we have disjoint union $\Delta_k^+ = \Phi_{\gamma_k}^+ \cup \{\gamma_k\} \cup \Delta_{k+1}^{+}$.

If  $\Delta_{k+1}=\emptyset$, then our sequences are $\{\Delta_i\}_{i=1}^k$, $\{\Delta_i^+\}_{i=1}^k$ and $\{\gamma_i\}_{i=1} ^k$.

  If $\Delta_{k+1}\neq \emptyset$, then from \eqref{for Delta_{k+1}0} and Lemma \ref{nsr3} a) it follows that $\Delta_{k+1}$
  is a root system and $\Delta_{k+1}^+ = \Delta_{k+1}\cap \Delta^+$.

By \eqref{for Delta_{k+1}0} and Lemma \ref{nsr3} b) applied to $\Delta_k$ and $\Delta_{k+1}$ it follows that  for any $\alpha \in \Delta_{k+1}^+$ we have
$
\{\beta \in \Delta_{k+1}^+\vert\ \alpha - \beta \in \Delta_{k+1}^+ \} =  \{\beta \in \Delta_k^+\vert\ \alpha - \beta \in \Delta_k^+ \}.$
 On the other hand by the induction assumption \eqref{induction ass 2} with $i=k-1$ it follows that for any $\alpha \in \Delta_{k+1}^+$ we have $
\Phi_\alpha^+ =  \{\beta \in \Delta_k^+\vert\ \alpha - \beta \in \Delta_k^+ \}.$ Hence for each $\alpha \in \Delta_{k+1}^+$ we have
$$\{\beta \in \Delta_{k+1}^+ ; \alpha-\beta \in \Delta_{k+1}^+ \} = \Phi_{\alpha}^+. $$
In particular we have proved \eqref{sol}.
The induction is complete, so we constructed a stem $\Gamma = \{\gamma_1, \dots,\gamma_d\}$.
\epr
From the proof of the last proposition we get some improvements of Proposition \ref{nsr6} and Lemma \ref{nsr3}
\begin{coro}\label{eu5} Let $\gamma_k, \Delta_k,\quad k=1,\dots,d$ be as in Proposition \ref{eu1}. Then

a) If $\alpha,\beta \in \Phi_{\gamma_{k}}^+$ and
$\alpha + \beta \in \Delta$, then $\alpha + \beta = \gamma_k$;

b) If $\alpha \in \Phi_{\gamma_{k}}^+ , \ \beta \in \Delta_{k+1}$ and
$\alpha + \beta \in \Delta$, then $\alpha + \beta \in \Phi_{\gamma_k}^+$;

c) If $\alpha, \beta \in \Delta_k$ and
$\alpha+\beta \in \Delta$, then $\alpha + \beta \in \Delta_k$.

\end{coro}
\bpr
c) Easy induction using Lemma \ref{nsr3}, b).

a) Because of c), we may apply  \eqref{sol} and Proposition \ref{nsr6}, a) to $\Delta_k^+$.

b) Now we apply \eqref{sol} and Proposition \ref{nsr6} b) to $\Delta_k^+$.
\epr
Now we can prove

\begin{theorem} \label{ir1}{\bf Existence and uniqueness}
Let $\Delta$ be a reduced root system, let $\Pi$ be a basis and let $\Delta^+$ be the corresponding set of positive roots. There exists exactly one stem of $\Delta^+$.
\end{theorem}
\begin{proof}
The existence is already proved in Proposition \ref{eu1}. Let $\Gamma =\{\gamma_1, \dots, \gamma_d\}$ be the stem of $\Delta^+$ constructed there. Now we prove uniqueness.

Let $\Gamma'$ be any stem of $\Delta^+$. We have to prove that $\Gamma = \Gamma'$. It is obviously sufficient to prove $\Gamma\subset\Gamma'$.

By maximality $\gamma_1 + \alpha $ is not a root for any $\alpha \in \Delta^+$ , so $\gamma_1 \not \in \Phi_\gamma^+$ for any $\gamma \in \Gamma' $, and because of \eqref{nos1} we have $\gamma_1 \in \Gamma'$.

Now assume that for some $k<d$ we have $\{\gamma_1,\dots,\gamma_k\}\subset \Gamma'$.  Assume that $\gamma_{k+1}\not \in \Gamma'$. Since $\Gamma'$ is a stem, there is an element $\delta \in \Gamma'$, such that $\gamma_{k+1} \in \Phi_\delta^+$.   Now $\delta \not \in \{\gamma_1,\dots, \gamma_k\}$ (since $\Gamma$ is a stem and $\gamma_{k+1} \not \in \Phi_{\gamma_1}^+\cup \dots \cup \Phi_{\gamma_k}^+$). Furthermore $\delta \not \in \Phi_{\gamma_1}^+\cup \dots \cup \Phi_{\gamma_k}^+$ because $\Gamma'$ is a stem. Therefore $\delta, \gamma_{k+1} \in \Delta_{k+1}^+$ and $\delta-\gamma_{k+1} \in \Delta^+$.

By Corollary \ref{eu5}, b)  it follows that $\delta-\gamma_{k+1} \not\in \Phi_{\gamma_i}^+$ for all $i\leq k$, hence $\delta-\gamma_{k+1} \in \Delta_{k+1}^+$. This is impossible by Corollary \ref{eu5}, c) and since $\gamma_{k+1}$ is a maximal root in $\Delta_{k+1}^+$. So $\gamma_{k+1}\in \Gamma'$. The theorem is proved.
\end{proof}

\begin{ex}\label{yyy1}
The root system $\Delta = \mk{D}_4$ is irreducible, and fixing $\Delta^+$ we determine a highest root $\gamma_1$, while $\Delta_2 = \mk{A}_1\oplus \mk{A}_1\oplus \mk{A}_1 $, so we have $\Gamma = \{\gamma_1,\gamma_2,\gamma_3,\gamma_4\}$, where the last three roots are all maximal and may come in any order.
\end{ex}
We fix a useful fact
\begin{coro}\label{uu1}
If $\gamma \in \Gamma$ and $\alpha \in \Phi^+_\gamma$, then $\alpha(H_\gamma) = C(\alpha,\gamma) = 1$.
\end{coro}
\bpr By the construction of the stem, $\gamma=\gamma_k$ is a
maximal root in a root subsystem $\Delta_k$, where obviously
$\alpha$ belongs to the branch at $\gamma$. Now we apply
Proposition \ref{gor6}, c).  \epr

From the construction in Proposition \ref{eu1} we obtain a natural ordering of the stem $\Gamma$ - there is a sequence  $\Delta_1\supset\Delta_2\supset\dots\supset\Delta_d$, which gives the indexation $\Gamma = \{\gamma_1,\dots,\gamma_d\}$. The ordering is substantially {\bf partial}. As the Example \ref{yyy1} shows, each time when $\Delta_k$ is not irreducible we have to {\bf choose}  $\gamma_{k+1}$  among the maximal roots of $\Delta_k^+$.
We shall give now the formal definition. First we have

\begin{prop}\label{yyy23}
Let $\Delta$ be a reduced root system, let $\Pi$ be a basis and let $\Delta^+$ be the corresponding set of positive roots. Let $\Gamma$ be the stem of  $\Delta^+$.

For each $\gamma \in \Gamma$, there exists an unique irreducible closed subsystem of roots $\Theta_\gamma \subset \Delta$, such that

a) The set $\Pi_\gamma = \Pi\cap\Theta_\gamma$ is a basis of $\Theta_\gamma$;

b) $\Theta_\gamma^+ = \Theta_\gamma\cap\Delta^+$ is the set of positive roots determined by the basis  $\Pi_\gamma$, $\gamma$ is the highest root of $\Theta_\gamma^+$ and
 $$
\Phi_\gamma^+ = \{\alpha \in \Delta^+\vert\ \gamma - \alpha  \in \Delta^+\} = \{\alpha \in \Theta^+_\gamma \vert\ \gamma - \alpha  \in \Theta^+_\gamma\}.
$$
c) The stem of $\Theta_\gamma^+ $ is the subset $ \Theta_\gamma\cap\Gamma$. If $\delta\in \Theta_\gamma\cap\Gamma$, then
$ \Theta_\delta \subset \Theta_\gamma$.
 \end{prop}
\bpr
We look at the proof of Proposition \ref{eu1}. If $\gamma = \gamma_k$ in the  construction there, then $\gamma$ is a maximal root in the reduced root system $\Delta_k$, which means exactly that $\gamma$ is the highest root of exactly one irreducible component of $\Delta_k$, which we denote by $\Theta_\gamma$. The check of the properties a), b) and c) is immediate.
\epr

\begin{df}\label{yyy2}
Let $\Delta$ be a reduced root system, let $\Pi$ be a basis and let $\Delta^+$ be the corresponding set of positive roots. Let $\Gamma$ be the stem of  $\Delta^+$ and let $\gamma, \delta \in \Gamma$.

We shall say that $\gamma$ {\bf precedes} $\delta$, and we write $\gamma \prec \delta$, if $\delta\in \Theta_\gamma$, (see Proposition \ref{yyy23}).
\end{df}

In the following text, each time when we use {\bf indexation} of $\Gamma$ we shall assume that it is {\bf compatible} with the  {\bf order $\prec$}, that is, when we write
$\Gamma = \{\gamma_1,\dots,\gamma_d\}$, {\bf we assume that}
\begin{gather}\label{yyy21}
\gamma_k \prec \gamma_j \Lw k < j.
\end{gather}

We illustrate the importance of the order in $\Gamma$ by the following useful
\begin{coro}\label{yyy7}
Let $\Gamma = \{\gamma_1,\dots,\gamma_d\}\subset \Delta^+$ be the stem of $\Delta^+$. Then for $i=1,\dots, d$ and $ \alpha \in   \Phi_{\gamma_i}^+$ we have
\begin{gather}
\label{ord2}    1\leq p < i \Lw  \alpha \pm  \gamma_p \not \in \Delta , \\
\label{ord1}   \alpha + \gamma_i \not \in \Delta,\quad \alpha - \gamma_i = s_{\gamma_i}(\alpha) \in \Phi_{\gamma_i}^-\\
\label{ord3}  i < p \leq d \mbox{ and }\beta \in \Delta_p \mbox{ and } \alpha + \beta \in \Delta \ \Lw  \ \alpha + \beta \in \Phi_{\gamma_i}^+.
\end{gather}
\end{coro}
\bpr
 All statements are direct consequences of the construction in Proposition \ref{eu1}, the properties in Corollary \ref{eu5} and Theorem \ref{ir1}.
 \epr
 The following corollary and remark are easy to verify and will be used freely in the rest of the paper.

\begin{coro}\label{yyy4} We use the notations of Proposition \ref{eu1}.

a) If $\Gamma = \{\gamma_1,\dots,\gamma_d\}$ is the stem of $\Delta^+$, then for each $k=1.\dots,d$ the stem of $\Delta_k^+$ is $\Gamma_k = \{\gamma_k,\dots,\gamma_d\} = \Gamma\cap \Delta_k^+$.

b) If we have a decomposition into orthogonal components $\Delta = \Delta'\cup\Delta''$ and $\Delta^+ = \Delta'^+\cup \Delta''^+$, then
$\Gamma = \Gamma'\cup\Gamma''$, where $\Gamma' = \Delta'\cap\Gamma$ and $\Gamma'' = \Delta''\cap\Gamma$. The order $\prec$ is induced in both directions.
\end{coro}

\begin{remark}\label{yyy3}
Let $\Delta$ be irreducible, let $\Pi$ be our fixed basis and let
$\nu\in {\bf Aut}_{\Pi}(\Delta)$ be a diagram automorphism. So $\nu(\Delta^+) = \Delta^+$ and if $\Gamma$ is the stem of $\Delta^+$, then obviously $\nu(\Gamma)$ is also a stem. By uniqueness (see Theorem \ref{ir1}) we have $\nu(\Gamma) = \Gamma$. Also, because $\nu$ is an automorphism and $\nu(\Delta^+) = \Delta^+$ we have $\nu(\Phi_\gamma^+) = \Phi^+_{\nu(\gamma)}$.

Moreover, if $\gamma$ is the highest root, then $\nu(\gamma) = \gamma$.
\end{remark}

Also from the construction in Proposition \ref{eu5} and Proposition \ref{gor2}, c) we get
\begin{coro}\label{yyy11}
The stem $\Gamma$ is a maximal strongly orthogonal subset of
$\Delta^+$.
\end{coro}
\bpr
Let $\gamma_p,\gamma_q \in \Gamma,\quad p < q$, then by construction $\gamma_q \in \Delta_{p+1} \subset \gamma_p^\perp$ and $\gamma_p$ is long in $\Delta_P$,
 whence we may apply Proposition \ref{gor2} c) to obtain strong orthogonality. From the definition of stem (formula \eqref{nos1}) it
  follows that no root may be strongly orthogonal to all $\gamma\in \Gamma$.
\epr

\begin{remark}\label{roo1}
A stem $\Gamma$ is a {\bf strongly orthogonal subset of $\Delta^+$ with maximal number of elements}, that is the number of elements
 of any strongly orthogonal subset $\Theta \subset \Delta$ is less or equal to the number of elements of $\Gamma$.
This fact is easy to prove and also easy to check comparing the list of stems of irreducible root systems in Section \ref{ex} with
the list of maximal strongly orthogonal subsets of irreducible root systems in \cite{AgaKan002}. We shall not use it in this paper.

It makes sense to notice that the converse is not true in general. For example, when $\Delta = \mk{A}_n$ (see Example \ref{fraka} below),
there are many different maximal strongly orthogonal subsets of $\Delta^+$, one of them is the stem.
Each of the others is the stem for some other choice of Weyl chamber.

On the other hand, if  $\Delta = \mk{C}_n$ (see Example \ref{frakc}) then the stem is the set of all long roots in $\Delta^+$.
It is the unique strongly orthogonal subset of $\Delta^+$ with maximal number of elements. In this case one and the same set $\Gamma$
is the stem of $\Delta^+$ for $n!$ different choices of the positive Weyl chamber. However the stem $\Gamma$ and the partial order $\prec$
 in it (see Definition \ref{yyy2}) determine $\Delta^+$ completely. The same holds in general.
\end{remark}

\begin{theorem}\label{yyy81}
Let $\Delta$ be a reduced root system, let $\Delta^+$ be a system of positive roots, let $\Gamma$ be the stem
 of $\Delta^+$ and let $\prec$ be the order in $\Gamma$ (see Definition \ref{yyy2}). Then the couple $(\Gamma,\prec)$
  determines $\Delta^+$.  Hence the Weyl group ${\bf W}$ acts simply transitively on the set of couples $(\Gamma,\prec)$.
\end{theorem}
\bpr
Let $\gamma_1,\dots,\gamma_d$ be any indexation of $\Gamma$ compatible with $\prec$. Then
\begin{gather*}
\Delta^+ = \{\alpha \in \Delta \vert\ C(\alpha,\gamma_1)=\dots=C(\alpha,\gamma_{k-1}) = 0,\ C(\alpha,\gamma_k)>0   \\
 \mbox{ for some  }k \in \{1,\dots,d\} \}.
\end{gather*}
Indeed, if $\alpha \in \pm\Gamma$ the above follows from strong orthogonality. If $\alpha \in \Phi$ by \eqref{nos1} there is
 exactly one $k \in \{1,\dots,d\}$, such that either $\alpha \in \Phi_k^+$ or $-\alpha \in \Phi_k^+$.
 By \eqref{ord2} and Proposition \ref{gor6} c) (applied to $\Delta_j$) we see that for $1\leq j < k$ we have $ C(\alpha,\gamma_j) = 0$.
 Then using \eqref{ord1} we see that $\alpha \in \Phi_k^+ \subset \Delta_+$ iff $C(\alpha,\gamma_k)>0$. The theorem is proved.
\epr

The stem decomposition \eqref{ste1} determines an useful involution on $\Delta^+$. We define
\begin{df}
\begin{gather}\label{mug1}
\mu(\alpha) =
\begin{cases}
\alpha & \mbox{ if } \alpha \in \Gamma, \\
-s_\gamma(\alpha) & \mbox{ if } \alpha \in \Phi_\gamma^+,\ \gamma \in \Gamma.
\end{cases}
\end{gather}
\end{df}
Thus, for $\alpha \in \Phi_\gamma^+$ we have $\alpha + \mu(\alpha) = \gamma$.

\subsection{The stem subalgebra}\label{ct}

We introduce notation for the Lie algebra entities which correspond to the root system combinatorics of the preceding subsection. So now $\mk{u}$ is a compact Lie algebra, $\mk{g} = \mk{u}^\CC$ is a reductive Lie algebra, $\mk{h}$ is a $\tau$-invariant Cartan subalgebra of $\mk{g}$ and $\Delta$ is the root system of $\mk{g}$ w. r. to $\mk{h}$. We fix a basis $\Pi$ of $\Delta$ so we have a fixed $\Delta^+$, a corresponding Borel subalgebra $\mk{b}^+$ etc. We shall always denote by $\Gamma = \{\gamma_1,\dots,\gamma_d\}$ the stem of $\Delta^+$.
 \begin{df}\label{di1}
 Let $\Gamma$ be the stem of $\Delta^+$. We denote
 \begin{gather*}
 \mc{V}_\gamma^\pm = span_\CC \{E_\alpha \vert\; \alpha \in \Phi_\gamma^\pm\},\quad \mc{V}_\gamma = \mc{V}_\gamma^+\oplus \mc{V}_\gamma^-,\quad \mc{V}_\gamma^\mk{u} = \mc{V}_\gamma\cap\mk{u};\\
\mc{V}^\pm = \bigoplus_{\gamma\in \Gamma}\mc{V}_\gamma^\pm,\quad \mc{V} = \mc{V}^+\oplus \mc{V}^-,\quad \mc{V}^\mk{u} = \mc{V}\cap\mk{u},\\
\mk{f} = \bigoplus_{\gamma\in \Gamma}sl_\gamma(2),\quad  \mk{f}^\pm = span_\CC\{E_{\pm \gamma}\vert \gamma \in \Gamma\},\quad \mk{f}_\mk{u} = \mk{f}\cap\mk{u};\\
\mk{o} = \bigcap_{\gamma \in \Gamma}\{H\in\mk{h} \vert \ \gamma(H)  = 0\},\quad   \mk{o}_s =   \mk{o}\cap  \mk{h}_s,  \quad \mk{o}_u =   \mk{o}\cap  \mk{u}.
\end{gather*}
We shall call the subalgebra $\mk{f}$ defined above {\bf the stem subalgebra}. The corresponding subgroup of ${\bf G}_s$ will be denoted by ${\bf F}$ and will be called {\bf the stem subgroup}.
\end{df}
If $\Gamma = \{ \gamma_1,\dots,\gamma_d\} $, in order  to simplify notation  we shall write sometimes
\begin{gather*}
H_k = H_{\gamma_k},\ E_k = E_{\gamma_k},\  su_k(2) = su_{\gamma_k}(2),\ \mc{V}_k  = \mc{V}_{\gamma_k},\  \mbox{ etc}.
\end{gather*}
In the language of reductive Lie algebras, the stem decomposition
\eqref{nos1} gives a decomposition of $\mk{n}^+$ into two step
nilpotent subalgebras.
\begin{df}\label{hei1}
Let $\gamma \in\Gamma$. We denote $\mk{heis}_\gamma = \mk{g}(\gamma)\oplus \mc{V}_\gamma^+$. We shall call $\mk{heis}_\gamma$ the
$\gamma$-{\bf component of} $\mk{n}^+$.
\end{df}
\begin{prop}\label{yyy5}
Let $\gamma \in\Gamma$. Then $\mk{heis}_\gamma$ is a Heisenberg algebra. We have a decomposition
$$\mk{n}^+ = \mk{heis}_{\gamma_1} \oplus\dots\oplus\mk{heis}_{\gamma_d}.
$$
\end{prop}
\bpr
By Corollary \ref{eu5}, a), all brackets in $\mk{heis}_\gamma$ vanish except
$$[E_\alpha,E_{\mu(\alpha)}] = N_{\alpha,\mu(\alpha)}E_\gamma.$$
The direct sum decomposition (of vector spaces of course) follows readily from \eqref{nos1}.
\epr
\begin{prop}\label{www1}
Let $\gamma \in \Gamma,\ \alpha \in \Phi_\gamma^+,\ \beta = \gamma - \alpha$. Then
$$
|N_{\gamma,-\alpha}| = 1,\quad N_{\gamma,-\alpha}N_{\gamma,-\beta} = -1.
$$
\end{prop}
\bpr Proposition \ref{gor6}, c) implies that under our assumptions
one has $p=0$ in formula  \eqref{hel1}, whence the first equality.

The second equality follows from the fact that  $\norm{\alpha} = \norm{s_\gamma(\alpha)} = \norm{\beta}$ and the formula
 \begin{gather} \label{fo1}
 \frac{N_{\alpha,\beta}}{\scal{\gamma,\gamma}} = \frac{N_{\beta,-\gamma}}{\scal{\alpha,\alpha}} =
 \frac{N_{-\gamma,\alpha}}{\scal{\beta,\beta}},\qquad \alpha,\beta,\gamma \in \Delta, \quad \alpha + \beta = \gamma.
\end{gather}
proved in (\cite{Helgason78}, ch III).  \epr

Now we return to the stem subalgebra. Because $\Gamma$ is strongly orthogonal, we have a decomposition
$\mk{f}_u = su_1(2)\oplus\dots\oplus su_d(2)\ $ into commuting subalgebras.  Now we shall introduce convenient bases for $\mk{f}_u$.
\begin{df}\label{we1}
 For $\gamma \in \Gamma$ we choose a $\rho_\gamma \in \CC,\ |\rho_\gamma|=1$. We denote $\rho = \{\rho_\gamma \vert\ \gamma \in \Gamma\}$ and
\begin{gather*}
W_\gamma = \frac{\ri}{2}H_\gamma,\quad  X_\gamma(\rho) = \frac{1}{2}(\rho_\gamma E_{\gamma} - \ol{\rho_\gamma} E_{-\gamma}),\\ Y_\gamma(\rho) =  X_\gamma(\ri\rho_\gamma) = \frac{\ri}{2}(\rho_\gamma E_{\gamma} + \ol{\rho_\gamma} E_{-\gamma});\\
\mk{w} = span_\RR\{W_\gamma\vert \gamma \in \Gamma\}, \quad \mk{x}(\rho) = span_\RR\{X_\gamma(\rho)\vert \gamma \in \Gamma\},\\ \mk{y}(\rho) = span_\RR\{Y_\gamma(\rho)\vert \gamma \in \Gamma\};\\
W_\Gamma = \sum_{\gamma\in\Gamma}W_\gamma,\quad X_\Gamma = \sum_{\gamma\in\Gamma}X_\gamma,\quad Y_\Gamma = \sum_{\gamma\in\Gamma}Y_\gamma,  \quad E_{\pm \Gamma}= \sum_{\gamma\in\Gamma}E_{\pm \gamma }\\
sl_\Gamma(2,\CC) = span_\CC\{W_\Gamma,X_\Gamma,Y_\Gamma\},\quad su_\Gamma(2) = span_\RR\{W_\Gamma,X_\Gamma,Y_\Gamma\}.
\end{gather*}
  \end{df}
It is clear that the three dimensional simple subalgebra $sl_\Gamma(2,\CC)\subset \mk{g}$ is generated by
the semisimple element $H_\Gamma = - 2 \ri W_\Gamma \in \mk{h}$ and the nilpotent elements $E_\Gamma$, $E_{-\Gamma}$.
\begin{remark}\label{rer}
 In the formulas of this section we shall sometimes suppress the dependence on the torus parameter $\rho$, i.e.
  we shall  write $X_\gamma$ instead of $X_\gamma(\rho)$ or $\mk{x}$ instead of $\mk{x}(\rho)$ etc. We hope that no confusion for the reader comes from this. In any case we remark that the subalgebras $sl_\gamma(2,\CC)$ and hence the subalgebras $su_\gamma(2) = sl_\gamma(2,\CC)\cap\mk{u}$ do not depend on $\rho$.

  For the interpretation of the results of Section \ref{hs} it will be convenient to have done our computations and theorems in the presence of $\rho$ (see also Remark \ref{di7}).
 \end{remark}
Obviously, for any $\rho_\gamma$ with $|\rho_\gamma|=1$, the  elements $W_\gamma, X_\gamma(\rho), Y_\gamma(\rho)$ span $su_\gamma(2) \subset \mk{f}_\mk{u}$.
By strong orthogonality of $\Gamma$ we have three $\tau$-invariant Cartan subalgebras of $\mk{g}$:
\begin{gather}\label{pp1}
 \mk{h}_I = \mk{w}^\CC\oplus \mk{o},\quad  \mk{h}_K = \mk{x}^\CC\oplus \mk{o},\quad  \mk{h}_J = \mk{y}^\CC\oplus \mk{o}.
\end{gather}
 We note that the above direct decompositions are orthogonal.

Next we interpret the important Corollary \ref{yyy7} in Lie algebra language.
\begin{prop}\label{pp2}
If $\gamma \in  \Gamma$, then the subspace $\mc{V}_\gamma$ is a
representation of the stem subalgebra $\mk{f}$ under $ad$. We
denote it by ${\bf r}_\gamma:\mk{f} \lw sl(\mc{V}_\gamma)$. We
denote by the same letter the corresponding representation  ${\bf
r}_\gamma:{\bf F}_u \lw SU(\mc{V}_\gamma^u)$.

a) If $\gamma, \delta \in \Gamma$ , then the restriction of ${\bf r}_\gamma$ to $sl_\delta(2)$  may be nontrivial only if $\gamma \preceq \delta$. Moreover

b) If $\gamma \neq  \delta$, then ${\mathcal V}_\gamma^+$ and ${\mathcal V}_\gamma^-$ are invariant under the $ad$ representation of $sl_\delta(2)$;

c) The action of $sl_\gamma(2)$ on ${\mathcal V}_\gamma$ decomposes into  2-dimensional irreducible components:\quad
$span_\CC\{E_\alpha,E_{s_\gamma(\alpha)} \}, \quad \alpha \in \Phi^+_\gamma.$
\end{prop}
\bpr
See Corollary \ref{yyy7}.
\epr

 We shall need several explicit formulas, describing the action of one-parameter subgroups of the stem subgroup $\bf F$ in the $Ad$ representation.

 \begin{remark} \label{xi11}
Obviously for each $\gamma\in\Gamma$ we have $\tau(X_\gamma(\rho))= X_\gamma(\rho)$, so for each $t\in\RR$ we have
$
\exp(\rm{t} ad X_\gamma)\circ \tau =  \tau \circ  \exp(\rm{t} ad X_\gamma).
$
Thus $\exp(\rm{t} ad X_\gamma)$ preserves $\mk{u}$.

Strong orthogonality of $\Gamma$ immediately implies that if $\gamma,\delta\in \Gamma,\ s,t\in\RR$, then
\begin{gather*}
 \exp({\rm t} ad X_\gamma)\circ \exp( {\rm s} ad X_\delta) =  \exp({\rm s} ad X_\delta)\circ  \exp({\rm t} ad X_\gamma).
\end{gather*}
\end{remark}
\begin{prop} \label{xi1} Let $\gamma \in \Gamma,\quad t\in\RR$ and $H \in \mk{h}$. Then
\begin{gather*}
 \exp({\rm t} ad X_\gamma)(W_\gamma) = \cos(t)W_\gamma - \sin (t)Y_\gamma,\\
\exp({\rm t} ad X_\gamma)(Y_\gamma) =  \sin (t)W_\gamma + \cos(t)Y_\gamma,\\
 \exp({\rm t} ad X_\gamma)(H) = H + \ri\gamma(H)(\sin(t)Y_\gamma + (1 - \cos(t))W_\gamma).
\end{gather*}
\end{prop}
\begin{proof} The three formulas follow by induction from the following:\\
$
ad X_\gamma(H)=  \ri\gamma(H)Y_\gamma,\quad ad X_\gamma(W_\gamma)= -Y_\gamma,\quad ad X_\gamma(Y_\gamma) = W_\gamma.$\end{proof}
 Because for $\gamma \in \Gamma, \ H\in \mk{o}$ we have $\gamma(H) = 0$, there is an obvious
 \begin{coro} \label{iexp11} Let $\gamma \in \Gamma,\ t\in \RR, \ H\in \mk{o}$. Then $\exp({\rm t} ad X_\gamma)(H)  =  H$.
\end{coro}
We note also
 \begin{coro} \label{iexp3} Let $\gamma \in \Gamma$. Then
\begin{gather*}
\exp({\rm t} ad X_\gamma)(E_\gamma)  =  E_\gamma - \ri\ol{\rho_\gamma}((\cos(t) -1)Y_\gamma + \sin (t)W_\gamma ).
\end{gather*}
\end{coro}
\begin{proof} Follows from the trivial $\exp(\rm{t} ad X_\gamma)(X_\gamma) = X_\gamma$
 and the formula for $Y_\gamma$ in Proposition \ref{xi1}.
\end{proof}
\begin{prop} \label{cay3} Let $\gamma \in \Gamma,\ \alpha \in \Phi_\gamma^+$. Then
\begin{gather*}
\exp({\rm t} ad X_\gamma)(E_\alpha) = \cos(\tfrac{t}{2})E_\alpha + N_{\gamma,-\alpha} \ol{\rho_\gamma}\sin(\tfrac{t}{2}) E_{s_\gamma(\alpha)}.
\end{gather*}
\end{prop}
\begin{proof} By induction, using Proposition \ref{www1}, for $n \geq 0$ we have
\begin{gather*}
(ad X_\gamma)^{2n+1}(E_\alpha) = \frac{(-1)^n}{2^{2n+1}} \ol{\rho_\gamma}N_{\alpha,-\gamma}E_{s_\gamma(\alpha)},\quad
(ad X_\gamma)^{2n}(E_\alpha) = \frac{(-1)^n}{2^{2n}}E_\alpha,
\end{gather*}
whence the proposition follows by summation of the series.
\end{proof}

 \subsection{The opposition involution}\label{di}
  \begin{df}\label{pni}
 Let $\Gamma = \{\gamma_1,\dots,\gamma_d\}$ be the stem of $\Delta^+$. For $\gamma \in \Gamma$ we denote
 $$
 \phi_\gamma = \phi_\gamma[\rho] = \exp(\pi ad X_\gamma(\rho)) \in {\bf Ad}(\mk{g}).
 $$
To simplify notations, for $k = 1,\dots,d$ we write $\phi_k = \phi_{\gamma_k}[\rho]$ and define
\begin{gather*}
\phi = \phi[\rho] = \phi_1\circ\dots\circ\phi_d = \exp(\pi ad X_\Gamma(\rho)).
\end{gather*}

\end{df}

\begin{remark}\label{di7}
It is well known (see e.g. Tits \cite{Tits66}), that if $\gamma \in \Delta,\ \rho \in \CC^\times = \CC\setminus\{0\}$ and
we define $X_\gamma = \frac{1}{2}(\rho E_{\gamma} - \frac{1}{\rho} E_{-\gamma})$, then the inner automorphism $exp(\pi ad X_\gamma)$
 is an extension of $s_\gamma$ (the reflection along $H_\gamma$ in $\mk{h}$) to an automorphism of $\mk{g}$.
 We have $\tau(X_\gamma(\rho)) = X_\gamma(\ol{\rho}^{-1})$, whence $X_\gamma(\rho) \in \mk{u} \iff |\rho| = 1$.

The reflections $\{s_\gamma \vert \gamma\in \Gamma\}$ generate an abelian subgroup ${\bf W}_\Gamma \subset {\bf W}$, which is obviously
isomorphic to $\ZZ_2\times\dots\times\ZZ_2$ ($d$ factors).

If we stay in the root system $\Delta$, the point of this subsection is the fact that for any choice of $\Delta^+$, hence of $\Gamma$, the product $s_{\gamma_1}\circ\dots\circ s_{\gamma_d}$ is the opposition element in the Weyl group of $\Delta$.
 However for our purposes we need to make explicit choice of a representative of the coset  $s_{\gamma_1}\circ\dots\circ s_{\gamma_d}$ in the exact sequence \eqref{elm1}.
\end{remark}

We recall that we denote by the same letter an automorphism $\psi \in {\bf N}(\mk{h}) \subset {\bf Aut}(\mk{g})$,
its action on $\mk{h}$ as an element of the Weyl group, and the conjugate action
on $\mk{h}^*$ given by $\psi(\alpha)(H) = \alpha(\psi^{-1}(H))$. In particular  from the third formula of Proposition \ref{xi1}
and the last remark we see that for each $\gamma$ we have $\phi_\gamma[\rho]\in {\bf N}_u(\mk{h})$ .

\begin{prop}\label{di4} The automorphism $\phi$ represents the "opposition involution" in the Weyl group, that is $\phi(\Delta^+) = \Delta^-$.
\end{prop}
\bpr  The properties of $\Gamma$ from Corollaries \ref{yyy7} and \ref{yyy11} give even more precise formulas. For each $k,j = 1,\dots,d,\ k\neq j$ we have
\begin{gather}\label{ccc1}
\phi_k(\gamma_k) = -\gamma_k,\quad \phi_k(\gamma_j) = \gamma_j,\quad \phi_k(\Phi_k^+)= \Phi_k^-, \quad \phi_k(\Phi_j^+)= \Phi_j^+.
\end{gather}
The proposition is proved.
\epr

\begin{prop}\label{elm13} If $H\in\mk{h}$, then
\begin{gather*}
\phi(H)= H - \sum_{\gamma\in\Gamma}\gamma(H)H_\gamma.
\end{gather*}
\end{prop}
\begin{proof}
The proposition follows from the third formula in Proposition \ref{xi1} and strong orthogonality of $\Gamma$ ($\gamma(H_\delta) = 0 $ if $\gamma\neq\delta$).
\end{proof}
Proposition \ref{elm13} has an obvious consequence.
\begin{coro}\label{elm14} We have
$$\mk{o} = \{H\in\mk{h}\vert\ \phi(H)= H\};\quad \mk{w}^\CC = \{H\in\mk{h}\vert\ \phi(H)= -H\}.$$
\end{coro}

From Proposition \ref{cay3} with $t = \pi$ we obtain
\begin{gather}\label{uu3}
 \phi_\gamma(E_\alpha)= \ol{\rho_\gamma}N_{\gamma,-\alpha}E_{s_\gamma(\alpha)},\quad \gamma\in\Gamma,\,\alpha\in\Phi^+_\gamma
 \end{gather}

\begin{df}
Let $\theta$ be the contragredience automorphism of $\mk{g}$ w.r. to $\mk{h}$(see \eqref{star2}), and let $\phi \in {\bf W}$ be the opposition automorphism of $\mk{g}$ w.r. to $\mk{h}$ (see Definition \ref{pni}). We denote
$$ \star =  \theta\circ\phi = \phi\circ\theta \in {\bf Aut}(\mk{g}).$$
We denote by the same symbol the adjoint involution $\star \in {\bf Aut}(\mk{h}^*)$.
\end{df}

It is well known that $\star \in {\bf Aut}_\Pi(\Delta)$ and
that when $\Delta$ is a reduced  irreducible root system then the involution $\star$ is nontrivial only when $\Delta = \mk{A}_n, n>1,\quad \Delta = \mk{D}_{2n+1}, n \geq 1,\quad \Delta = \mk{E}_6$.

\begin{coro}\label{elm15} We have
$$\mk{o} = \{H\in\mk{h}\vert\ \star(H)= -H\};\quad \mk{w}^\CC = \{H\in\mk{h}\vert\ \star(H)= H\}.$$
\end{coro}

\subsection{Back to the roots}
We are going to show (improving Proposition \ref{yyy8}) how $\star$ determines the number of elements of the stem:

\begin{prop} \label{star3} Let $\gamma \in \Gamma$. Then

a) $\star\gamma = \gamma,\quad \star(\Phi_{\gamma})^+  = \Phi_{\gamma}^+,\quad \star (\Phi_{\gamma}^+ \cap \Pi) = \Phi_{\gamma}^+ \cap \Pi$;

b) If $\gamma \in \Gamma$, then we have  a trihotomy:\qquad
i) $\gamma \in \Pi$ and $\Phi^+_\gamma = \emptyset$; \qquad ii) $\Phi^+_\gamma\cap\Pi$ has exactly one element;\qquad
iii) $\Phi^+_\gamma\cap\Pi$ has exactly two elements.

c) If $\alpha, \beta \in \Phi_\gamma^+ \cap \Pi$ and  $\alpha \neq \beta$, then $\star\alpha = \beta$.
\end{prop}
\bpr
The properties a) follow directly from Proposition \ref{di4} and  \eqref{ccc1}.

The trihotomy b) is Proposition \ref{yyy8} in the case when $\gamma$ is a maximal root of $\Delta^+$.
To prove it for any $\gamma \in \Gamma$ we just have to apply  Proposition \ref{yyy8} to
the closed root subsystem $\Theta_\gamma$ (see Proposition \ref{yyy23}), where $\gamma$ is the highest root.

We proceed to prove c). We are going to use the order from Definition \ref{yyy2} (see also the convention \eqref{yyy21}).
Let $\Gamma = \{\gamma_1,\dots,\gamma_d\}$ and let $
\gamma = \gamma_k \in \Gamma,\quad \Phi_k^+ \cap \Pi = \{\alpha, \beta\}$ and $\alpha \neq \beta$.

Any $\zeta \in \Delta$ decomposes as follows:
\begin{gather}\label{aim1}
 \zeta = \sum_{\lambda \in \Pi } n_\lambda(\zeta) \lambda.
\end{gather}
Obviously  $n_\alpha(\alpha) = 1$. In order to prove $\star\alpha = \beta$ it is sufficient to prove that $n_\alpha(\star\alpha) = 0$, we proceed to do this.

We have $\star\alpha = -s_{\gamma_d}\circ\dots\circ s_{\gamma_1}(\alpha) = -s_{\gamma_d}\circ\dots\circ s_{\gamma_k}(\alpha)$ because, by \eqref{ord2}, for $j =1,\dots,k-1$ the reflection $s_{\gamma_j}$ leaves $\Phi_k$ pointwise fixed.

Denote $ \zeta = s_{\gamma_k}(\alpha) \in \Phi_k$. Then $n_\alpha(\zeta) = n_\alpha(\alpha) - n_\alpha(\gamma_k) = 0$,
since $n_\alpha(\gamma_k)= 1$ (see the proof of Proposition \ref{yyy8}).

The proposition will be proved if we show that for any $\zeta \in \Phi_k$ and $j > k$ we have
$n_\alpha(\zeta) = n_\alpha(s_{\gamma_j}\zeta)$.  The last equation follows obviously from the fact that  for $j>k$ we have $n_\alpha(\gamma_j) = 0$.
Indeed, by definition (see Proposition \ref{eu1}) of $\gamma_j$ as maximal root of $\Delta_j$ we know that $n_\lambda(\gamma_j) \neq 0$
 only for $\lambda \in \Pi\cap\Delta_j$ (see \eqref{aim1}). By the definition of stem $\alpha \not \in \Delta_j$.
\epr
From Proposition \ref{star3} it is trivial to get
\begin{coro}\label{ggg1}
Let $\mk{g}$ be a semisimple Lie algebra, let $\Pi$ be a basis of $\Delta$ and let $\Gamma = \{\gamma_1,\dots,\gamma_d\}$ be the corresponding stem. Then
\begin{gather*}
 d = \#(\Pi/\{id,\star\}).
\end{gather*}
In particular $\tfrac{1}{2}rank(\mk{g}) \leq d \leq rank(\mk{g})$.
\end{coro}
\begin{ex}
If $\Delta = \mk{A}_n$ (see Example \ref{fraka}), then $d = \ksco{\tfrac{n+1}{2}}$, if $\Delta = \mk{C}_n$  (see Example \ref{frakc})  then $d = n$.
\end{ex}
\begin{coro} \label{star5} Denote $\widetilde{\Gamma} = \{ \gamma \in \Gamma\ \vert\ \Phi_{\gamma}^+ \cap \Pi = \{ \alpha_{\gamma}, \beta_{\gamma} \}, \alpha_{\gamma} \neq  \beta_{\gamma} \}$. Then
\begin{gather*}
 (\mk{o}_s)_{\RR}^* = \Gamma^\bot = span\{ \alpha_{\gamma}-\beta_{\gamma}\vert \gamma \in \widetilde{\Gamma} \} ;\qquad \star(\zeta) = -\zeta,\quad \zeta \in \Gamma^\bot.
\end{gather*}
\end{coro}
\bpr Obviously the set  $\{ \alpha_{\gamma}-\beta_{\gamma}; \gamma \in \widetilde{\Gamma} \}$ is linearly independent.

For any $\gamma \in \widetilde{\Gamma}$ and $\delta \in \Gamma$ by  Proposition  \ref{star3} we have $\star(\delta)=\delta$ and $\star(\alpha_\gamma) = \beta_{\gamma}$, hence, since $\star$ is isometry, $\scal{\alpha_{\gamma}, \delta} = \scal{\beta_{\gamma}, \delta}$. Therefore  $\scal{\alpha_{\gamma} - \beta_{\gamma}, \delta} = 0$.

By  Proposition \ref{yyy8} it follows that $\#(\Gamma) + \#(\widetilde{\Gamma}) = \#(\Pi)$ and the first formula follows.
The rest follows from $\star(\alpha_{\gamma})=\beta_{\gamma},\quad \gamma \in \widetilde{\Gamma}$.
\epr

\subsection{The Cayley transform}\label{cat}

We define an automorphism which is a square root of the opposition involution $\phi$ from the previous subsection, we use freely all notation introduced there.
\begin{df}\label{xi_p}
For $p=1,\dots, d$ we denote $X_p = X_{\gamma_p}(\rho)$. Let
\begin{gather*}
{\bf c}_p = {\bf c}_p[\rho] = \exp\left(\frac{\pi}{2} ad X_p(\rho) \right)  \in {\bf Ad}(\mk{g}),\\
  {\bf c}= {\bf c}[\rho] = {\bf c}_1 \circ {\bf c}_2 \circ \ldots \circ {\bf c}_d  = \exp\left(\frac{\pi}{2} ad X_\Gamma(\rho)\right).
  \end{gather*}
\end{df}
\begin{remark}\label{xxx1} By Remark \ref{xi11} we conclude that for $p = 1,\dots,d$ we have $ {\bf c}_p\circ \tau = \tau\circ{\bf c}_p$, so
all ${\bf c}_p$ and ${\bf c}$ are automorphisms of $\mk{u}$.

Also from Remark \ref{xi11} it follows that for $i,j= 1,\ldots,d$ we have ${\bf c}_i \circ {\bf c}_j = {\bf c}_j\circ {\bf c}_i$,
 whence the definition of the automorphism ${\bf c}$ does not depend on the order of the factors and that is why we may define it as exponent of one element, namely $\tfrac{\pi}{2}ad X_\Gamma$.

Obviously  for each $k = 1,\dots,d$ we have ${\bf c}_k^2 = \phi_k$ (see Definition \ref{pni}), whence $
{\bf c}^2 = \phi$, i.e. ${\bf c}$ is a square root of the opposition involution.

\end{remark}

We shall need an explicit description of the action of ${\bf c}$ on
$\mk{f}\oplus\mk{o}$.
\begin{prop}\label{elm2} If $\gamma\in\Gamma$, then $$\quad{\bf c}(X_\gamma) = X_\gamma,\quad {\bf c}(Y_\gamma) = W_\gamma,\quad {\bf c}(W_\gamma) = -Y_\gamma. $$
\end{prop}
\begin{proof}
The first equality obviously follows from strong orthogonality of $\Gamma$. The second and third formulas follow directly from Proposition \ref{xi1}.
\end{proof}

\begin{prop}\label{elm3} If $H\in\mk{h}$, then
\begin{gather}
{\bf c}(H)= H +\ri \sum_{j=1}^d
\gamma_j(H)(W_j + Y_j), \nonumber \\[-2mm]  \label{elm11} \\[-2mm] {\bf c}^{-1}(H)= H + \ri\sum_{j=1}^d
\gamma_j(H)(W_j - Y_j). \nonumber
\end{gather}
\end{prop}
\begin{proof}
The proposition follows from the third formula in Proposition \ref{xi1} and strong orthogonality of $\Gamma$ ($\gamma(W_\delta) = 0 $ if $\gamma\neq\delta$).
\end{proof}
Obviously \eqref{elm11} and Proposition \ref{elm2} give
\begin{gather}\label{elm12}
\mk{o} = \{H\in\mk{h} \vert\ {\bf c}(H) = H \}.
\end{gather}
\begin{remark}\label{mai3}
We may define
$${\bf c}_\mk{y} = \exp\left(\frac{\pi}{2} ad Y_\Gamma\right),\quad {\bf c}_\mk{w} = \exp\left(\frac{\pi}{2} ad W_\Gamma\right). $$
Writing for the sake of symmetry ${\bf c}_\mk{x}$ for the Cayley transform defined at the beginning of this subsection, we have (see Proposition \ref{elm2}):
\begin{gather}
{\bf c}_\mk{w}(\mk{w})= \mk{w},\quad {\bf c}_\mk{w}(\mk{x}) = \mk{y}.\qquad {\bf c}_\mk{x}(\mk{x})= \mk{x},\quad {\bf c}_\mk{x}(\mk{y}) = \mk{w}.\nonumber \\
{\bf c}_\mk{y}(\mk{y})= \mk{y},\quad {\bf c}_\mk{y}(\mk{w}) = \mk{x}. \nonumber
\end{gather}
The elements ${\bf c}_\mk{x}^2$ and ${\bf c}_\mk{y}^2$ represent the opposition involution w.r. to the Cartan subalgebra $\mk{h}_I$, also ${\bf c}_\mk{w}^2$ and ${\bf c}_\mk{y}^2$ represent the opposition involution w.r. to the Cartan subalgebra $\mk{x}^\CC \oplus \mk{o}$, etc..

In order to prove the statement in this remark, there is no need for new computations, actually we know, that putting $\ri\rho$ in the place of $\rho$ we change $X_\gamma$ to $Y_\gamma$ and $Y_\gamma$ goes to $-X_\gamma$ in all formulas of this section. The corresponding statements about ${\bf c}_\mk{w}$ are very easy to check.
\end{remark}

\section{Existence of a hypercomplex structure}\label{sus1}

Now we use the root combinatorics of the stem to find sufficient conditions for admissibility of a  $\mk{b}^+$ complex structure on $\mk{u}$.
We present our candidate for a match to $I$.
\begin{df}\label{qw4} Let $I$ be a $\mk{b}^+$ complex structure on $\mk{u}$. We denote
 $$J = I_{\bf c} = {\bf c}\circ{I}\circ{\bf c}^{-1}.$$
\end{df}

By definition $J$ is equivalent to $I$, so $J$ is an integrable ${\bf c}(\mk{b^+})$ complex structure on $\mk{u}$. We obviously have $\mk{m}^+_J = {\bf c}(\mk{m}^+_I)$.
Proposition \ref{elm2} and formula \eqref{elm12} imply
$\mk{h}_J = {\bf c}(\mk{h}_I) = \mk{y}^\CC \oplus \mk{o}$.

In this section we give a necessary and sufficient condition for $IJ = -JI$.

\begin{remark}\label{iii7}
By the definition of $\Phi^+_\gamma$ we see that $dim(\mc{V}^+_\gamma)$ is even, whence $dim(\mc{V}_\gamma)$ is divisible by 4, whence $dim_\CC(\mc{V})$ is divisible by 4. So $dim_\RR(\mc{V}^u)$ is divisible by 4.

From the decomposition $\mk{u} = \mc{V}^u\oplus\mk{f}_u\oplus\mk{o}_u$ we see that $dim(\mk{u})$ is divisible by 4 if and only if $3d + dim(\mk{o}_u)$ is divisible by 4. So in the following we shall always assume (sometimes implicitly) that $dim(\mk{o}_u) = d + 2p$, where $p$ is some even integer.
\end{remark}
From Proposition \ref{pp2} it follows trivially that $\mc{V}_\gamma$ is ${\bf c}$-stable.
\subsection{The structure J on \texorpdfstring{$\mc{V}$}{\space}}\label{mcV}

First we are going to prove that $J(\mc{V}_\gamma^+) = \mc{V}_\gamma^-$, whence $IJ = -JI$ holds on $\mc{V}$ without any further conditions.
We begin with
\begin{prop}\label{ord5} If $V \in \mc{V}_\gamma$, then
$ IV =  2 adW_\gamma(V)$.
\end{prop}
\begin{proof} If $\alpha \in \Phi_\gamma^+$, then by Corollary \ref{uu1} we have
\begin{gather*}
[2W_\gamma,E_\alpha] = \ri\alpha(H_\gamma)E_\alpha = \ri C(\alpha,\gamma)E_\alpha = \ri E_\alpha, \\
[2W_\gamma,E_{-\alpha}] = - \ri\alpha(H_\gamma)E_{-\alpha} = -\ri E_{-\alpha}.
\end{gather*}
The proposition is proved.
\end{proof}

We use Proposition \ref{ord5} to make the next step
\begin{prop} \label{bbb4} If $V \in \mc{V}_\gamma$, then
$\quad JV = - 2 adY_\gamma(V) = -\phi_\gamma[\ri\rho](V).$
\end{prop}
\bpr
If $\alpha \in \Phi_\gamma^+$, then  using Propositions \ref{ord5} we have
$$
JE_\alpha =
 {\bf c}I{\bf c}^{-1}E_\alpha = 2{\bf c}[W_\gamma,{\bf c}^{-1}E_\alpha] =  2[{\bf c}W_\gamma,E_\alpha].
$$
By Proposition \ref{elm2} and formula \eqref{uu3} (using
$Y_\gamma(\rho) = X_\gamma(\ri\rho)$) we get
\begin{gather}\label{uu4}
JE_\alpha = -[2Y_\gamma,E_\alpha] = \ri\ol{\rho_\gamma}N_{\gamma,-\alpha}E_{s_\gamma(\alpha)} = - \phi_\gamma[\ri\rho](E_\alpha).
\end{gather}
Further $JE_{-\alpha} = \tau(J\tau{E_{-\alpha}} )=  \tau([-2Y_\gamma,-E_\alpha]) = [-2Y_\gamma,E_{-\alpha}]$, whence the proposition follows.
\epr
From Proposition \ref{bbb4} it follows that
\begin{coro}\label{qw6} For each $\gamma\in\Gamma$ we have
 $J(\mc{V}_\gamma^+) = \mc{V}_\gamma^-$.
\end{coro}
At the end, from  Corollary \ref{qw6} (as in Lemma \ref{oo1}) we obtain
\begin{coro}\label{qw8} For each $V \in \mc{V}$ we have
 $IJV = -JIV$.
\end{coro}
\begin{remark}
In particular we have proved that for each $V\in \mc{V}_\gamma$ we have
$\quad adW_\gamma adY_\gamma (V) = -adY_\gamma adW_\gamma (V),$ whence
\begin{gather*}
adW_\gamma adY_\gamma (V) = \frac{1}{2}[ad W_\gamma,ad Y_\gamma](V) =  \frac{1}{2}ad[W_\gamma,Y_\gamma](V) =  - \frac{1}{2}adX_\gamma(V).
\end{gather*}
Denoting as usual $K=IJ$, for each $V\in \mc{V}_\gamma$ we have
\begin{gather}\label{uu5}
KV = IJV  = 2adX_\gamma(V) = \phi_\gamma(V).
\end{gather}
\end{remark}
\subsection{The complex structure \texorpdfstring{$J$}{\space} on \texorpdfstring{$\mk{f}\oplus\mk{o}$}{\space}}\label{fu1}

As explained in Subsection \ref{css} we have some freedom in defining a $\mk{b}^+$ complex structure $I$ on the Cartan subalgebra $\mk{h}$. While $IX_\gamma = Y_\gamma$ is fixed by the convention that $IX = \ri X \quad (IX = -\ri X )$ on $\mk{n}^+ \quad (\mk{n}^-) $, we have substantial freedom choosing the  elements $IW_\gamma\in \mk{h}_u$. At the end, it turns out that the necessary and sufficient condition for admissibility of $I$ is a condition on $I(\mk{w})$.
\begin{df}\label{eee81}
Let $I$ be a $\mk{b}^+$ complex structure on $\mk{u}$. For each $\gamma\in \Gamma$ we denote:
$$ Z_\gamma = IW_\gamma, \quad \mk{z} = \mk{z}_I = I(\mk{w}) = span_\RR\{Z_\gamma\vert\ \gamma\in \Gamma\}\subset \mk{h}_u.$$
We call the subalgebra $\mk{e} = \mk{e}_I = \mk{f}_u +  \mk{z}_I$ the {\bf extended stem subalgebra}.
\end{df}

First we compute the operator $J = I_{\bf c}$ on $\mk{w}$.
\begin{prop}\label{oo121}
For each $\gamma\in\Gamma$ we have $JW_\gamma = - X_\gamma$.
\end{prop}
\begin{proof}
By Proposition \ref{elm2} we compute\\ $JX_\gamma = {\bf c}\circ I\circ {\bf c}^{-1}X_\gamma =  {\bf c} Y_\gamma = W_\gamma$.
\end{proof}

\begin{prop}\label{la1} Let $I$ be a $\mk{b}^+$ complex structure on $\mk{u}$ and let $J = I_{\bf c}$.
The following three conditions are equivalent

a) $\mk{z} \subset \mk{o}$;

b) For each $\gamma\in\Gamma$ we have $JZ_\gamma = Y_\gamma$.

c) For $X\in \mk{f}_u + \mk{z}$ we have $IJX = - JIX$.
\end{prop}
\begin{proof}
a) $\Lw$ b). By  \eqref{elm12} and Proposition \ref{elm2} we have
$$
JZ_\gamma =  {\bf c}\circ I\circ {\bf c}^{-1}Z_\gamma =  {\bf c}\circ IZ_\gamma  = - {\bf c}W_\gamma  = Y_\gamma.
$$
b) $\Lw$ c) We use the definition of $I$ and Proposition \ref{oo121}. to compute
\begin{gather*}
IJW_\gamma = -I X_\gamma = -Y_\gamma = - JZ_\gamma = -JIW_\gamma;\\
 IJZ_\gamma = IY_\gamma = - X_\gamma = JW_\gamma =- JIZ_\gamma.
\end{gather*}
c) $\Lw$ a). From Proposition \ref{oo121} we get $JZ_\gamma = JIW_\gamma = -IJW_\gamma = I X_\gamma = Y_\gamma$. But then
$$
Z_\gamma = IW_\gamma = IJX_\gamma = - JIX_\gamma = -JY_\gamma = {\bf c}(IW_\gamma) = {\bf c}(Z_\gamma).
$$
So by \eqref{elm12} we have $Z_\gamma \in \mk{o}$.
\end{proof}
We collect the above results in the following
\begin{coro}\label{oo2} Let $I$ be a $\mk{b}^+$ complex structure on $\mk{u}$ and let $J = I_{\bf c}$.
Then $\mk{z}\subset \mk{o}$ if and only if
\begin{gather*}
IX_\gamma =  Y_\gamma,\quad IW_\gamma = Z_\gamma;\qquad JX_\gamma =  W_\gamma,\quad JZ_\gamma = Y_\gamma,\quad \gamma\in\Gamma.
\end{gather*}
\end{coro}
\bpr
The first three formulas obviously follow from the definition of $I$,  $Z_\gamma$ and $J = I_{\bf c}$ (see Propositions \ref{oo121}) . The last formula above was proved in Proposition \ref{la1} to be equivalent to $\mk{z}\subset \mk{o}$.
\epr
We also have obviously
\begin{coro}\label{uu2}
If $\mk{z}\subset \mk{o}$, then the extended stem subalgebra $\mk{e} = \mk{f}_u\oplus \mk{z}$ is invariant under $I,J$.
\end{coro}

We are now ready to prove
 \begin{prop}\label{ij1}
If $I(\mk{w}) = \mk{o}_u$, then $J = I_{\bf c}$ matches $I$.
\end{prop}
\bpr
In any case $J$ is integrable, because it is equivalent to $I$.

We have $\mk{g} = \mc{V} \oplus \mk{f}\oplus \mk{o}$ and under our assumption we have $\mk{f}\oplus\mk{o} = \mk{e}^\CC$. From Corollaries \ref{qw6} and \ref{oo2} respectively we get
$$
J(\mc{V}^u) = \mc{V}^u,\quad J(\mk{e}) = \mk{e}.
$$
 Now from Corollaries \ref{qw8} and \ref{oo2} we have $IJ = - JI$, on both direct summands. The theorem is proved.
\epr
\begin{remark}\label{ij31}
It is easy to see that the condition  $I(\mk{w}) = \mk{o}_u$ is equivalent to $2d = rank(\mk{u})$ which implies that $dim(\mk{u})$ is divisible by 4  (see Remark \ref{iii7}). When $2d =
rank(\mk{g})$, any complex structure $I$ on $\mk{h}_u$ with $I(\mk{w}) = \mk{o}_u$ extends in an obvious way to a $\mk{b}^+$  complex structure on $\mk{u}$. Thus by  Proposition \ref{ij1}, if ${\bf U}$ is a compact Lie group  such that $2d = rank(\mk{u})$, then ${\bf U}$ carries a left invariant hypercomplex structure.
\end{remark}
In order to state the sufficient condition in the general case we
need some more notation.
 \begin{df}\label{mj1}
Let $\Gamma$ be the stem of $\Delta^+$ and let $\mk{z} \subset \mk{o}$.  We denote
\begin{gather*}
P_\gamma = W_\gamma - \ri Z_\gamma,\quad Q_\gamma = W_\gamma + \ri Z_\gamma = \tau{P_\gamma},\qquad \gamma \in \Gamma;\\
 \mk{v} = (\mk{w}\oplus \mk{z})^\CC,\quad \mk{v}^+ = \mk{v}\cap\mk{h}^+_I,\quad \mk{v}^- = \mk{v}\cap\mk{h}^-_I \quad \mk{v}_u = \mk{w}\oplus \mk{z};\\
\mk{j}^+  = \mk{o}\cap \mk{h}^+,\quad \mk{j}^- =  \mk{o}\cap \mk{h}^-,\quad
\mk{j} = \mk{j}^+ \oplus \mk{j}^- ,\quad \mk{j}_u = \mk{j}\cap\mk{u}.
\end{gather*}
\end{df}
\begin{prop}\label{mj5} Let $I(\mk{w}) \subset \mk{o}_u$. Then $$\mk{v}^+ = span_\CC\{P_\gamma\vert\ \gamma\in \Gamma\},\quad \mk{v}^- = span_\CC\{Q_\gamma\vert\ \gamma\in \Gamma\}.$$ We have
\begin{gather}\label{mj7}
\mk{h}  = \mk{v}\oplus \mk{j},\quad \mk{h}^+  = \mk{v}^+\oplus \mk{j}^+,\quad\mk{h}^- = \mk{v}^-\oplus \mk{j}^-,\quad \mk{o}_u = \mk{z}\oplus\mk{j}_u.
\end{gather}
\end{prop}
\bpr
The condition $I(\mk{w}) \subset \mk{o}$ implies  $\gamma_k(P_j) = \ri \delta_{k,j}$, so $P_1,\dots,P_d$ is a basis of $\mk{v}^+ \subset \mk{h}^+$. On the other hand $\mk{j}^+ = \{H\in \mk{h^+}\vert\ \gamma_1(H) = \dots = \gamma_d(H) = 0\}$, thus $\mk{h}^+ = \mk{v}^+\oplus\mk{j}^+$. In the same way
$\mk{h}^- = \mk{v}^-\oplus\mk{j}^-$.

Now, in order to prove $\mk{h}  = \mk{v}\oplus \mk{j}$, we have to  show only that $\mk{v}\cap \mk{j}=\{0\}$.
Let $ X \in \mk{v} \cap \mk{j} $. We may decompose $X = X^+ + X^-$, where $X^+ \in \mk{v}^+, X^- \in \mk{v}^-$. Obviously $I(\mk{j})=\mk{j}$, hence $I(X)= \ri X^+ - \ri X^- \in \mk{j}$. The inclusions  $X^+ + X^- \in \mk{j},  \ri X^+ - \ri X^- \in \mk{j}$ imply $X^+ \in \mk{j}$, $X^- \in \mk{j}$, therefore  $X^+ \in \mk{j}^+ \cap \mk{v}^+$, $X^- \in \mk{j}^- \cap \mk{v}^-$. Now from $\mk{v}^+\cap \mk{j}^+ = \mk{v}^-\cap\mk{j}^-=\{0\}$ we obtain $X^+=X^-=0$ .

\epr

\begin{prop}\label{mj2} Let $I(\mk{w}) \subset \mk{o}$ and let $rank(\mk{g}) = 2d + 2p$, where $p$ is some even nonnegative integer. We have $dim_\RR(\mk{j}_u) = 2p$. Also\\
a) $I(\mk{j}) = \mk{j},\quad I(\mk{j}_u) = \mk{j}_u$;\\
b) If $H \in \mk{j}$, then $I_{\bf c}H = IH$.
\end{prop}
\bpr

If $H \in \mk{j}$ , then $H = A + B,\quad A \in \mk{j}^+,\ B \in \mk{j}^-$. Thus $IH = \ri A - \ri B \in \mk{j}$. For the second equality, note that $\mk{u}$ is also invariant under $I$. So a) is proved.

Because $\mk{j} \subset \mk{o}$, for $H \in \mk{j}$  \eqref{elm12} implies  ${\bf c}^{-1}H = H$, then by a) of this proposition we have $I{\bf c}^{-1}H = IH \in \mk{j}$ and again by \eqref{elm12} we have ${\bf c}IH = IH$. Thus, item b) is proved.
\epr

We have the following important
\begin{remark}\label{mj3}

Note that the extended stem subalgebra $\mk{e}$ (see Definition \ref{eee81}) is closed under the action of $I, I_{\bf c}$. The corresponding subgroup ${\bf E}_u$ may {\bf not be a closed subgroup} of ${\bf U}$. If  ${\bf E}_u$  is a closed subgroup, which is an arithmetic condition on the $Z_\gamma$ (vacuously fulfiled when $\mk{u}$ is nearest to semisimple), then ${\bf E}_u$ is a hypercomplex submanifold of ${\bf U}$.

Obviously also $\mk{e}^\CC =\mk{f}^+ \oplus \mk{f}^- \oplus \mk{v}$ is always a subalgebra  of $\mk{g}$ invariant under the action of $I,J$ ( the complexified extended stem subalgebra ).

The subspace $\mc{V}\oplus\mk{v} = \mk{n}^+ \oplus \mk{n}^- \oplus \mk{v}$  (and $\mc{V}^u \oplus \mk{v}_u$) is also invariant under the action of $I,J$, but is {\bf not} obliged to be {\bf a subalgebra}.  An example is $\mk{g} = sl(3,\CC)\oplus \mk{c}$ where $\mk{c} \cong \CC^4$. Then $\Gamma = \{\gamma\}$, we may take $IW_\gamma \in \mk{c}_u$ so  $\mk{v} = span_\CC\{P_\gamma,Q_\gamma\}$ does not contain $\mk{h}_s$
\end{remark}

\begin{theorem}\label{suf1}
Let $\mk{u}$ be a compact Lie algebra, whose dimension is divisible by 4, and let $I$ be a $\mk{b}^+$ complex structure on $\mk{u}$. If $\mk{z} = I(\mk{w})\subset \mk{o}_u$, then $I$ is admissible.
\end{theorem}
\bpr
We have a decomposition (of real vector spaces) $\mk{u} = \mc{V}^u \oplus \mk{f}_u\oplus \mk{z}\oplus \mk{j}_u$
\begin{gather}\label{iii1}
\mk{m}^+_I = \mk{n}^+ \oplus \mk{v}^+ \oplus \mk{j}^+.
\end{gather}

Let $S_1,\dots,S_p$ be a basis of $\mk{j}^+$. Define $T_k = \tau(S_k),\ k=1,\dots,p$, then $T_1,\dots,T_p$ is a basis of $\mk{j}^-$. Let ${\bf b}$ be any $p\times p$ complex matrix such that ${\bf b}\ol{\bf b} = -1$. Then we may define a complex structure $B$ on $\mk{j}_u$ by
\begin{gather}\label{iii2}
BS_j = \sum_{k=1}^p b_{k,j}T_k,\quad  j = 1,\dots, p.
\end{gather}
Obviously $(I,B)$ define a quaternionic structure on the vector space $\mk{j}_u$, whence we may decompose the  $\mk{j}= \mk{j}_B^+ \oplus \mk{j}_B^-$ into the $\ri$ and $-\ri$ eigenspaces of $B$ respectively.

Now we may define a matching complex structure $J$ on $\mk{u}$:
 \begin{gather}\label{iii3}
JX =
\begin{cases}
I_{\bf c}X& \mbox{ if } X  \in \mc{V}^u \oplus \mk{f}_u \oplus \mk{z},\\
BX &  \mbox{ if } X \in \mk{j}_u.
\end{cases}
\end{gather}
Obviously we have $IJ + JI = 0$. To show that $J$ is integrable we note that $J$ is a regular complex structure w.r. to the Cartan subalgebra $\mk{y}^\CC \oplus \mk{o}$ (see \eqref{elm12}). More explicitely from \eqref{elm12}, \eqref{iii2} and Corollary \ref{oo2} we have
$$
\mk{m}_J^+ = \mk{j}_B^+ \oplus span_\CC\{Y_\gamma + \ri Z_\gamma\vert\ \gamma \in \Gamma\}\oplus {\bf c}(\mk{n}^+) ={\bf c}(\mk{j}_B^+ \oplus \mk{v}^+  \oplus \mk{n}^+ ),
$$
which is a subalgebra. Certainly $\mk{n}(\mk{m}_J^+) = \mk{b}_J^+ = {\bf c}(\mk{b}_I^+)$.
\epr

\begin{remark}\label{oo3}
In the classic description of quaternions we have a third complex structure $K = IJ$. In order to get it we should have used another Cayley transform (see Remark \ref{mai3})
$$
{\bf c}_\mk{y} = \exp \sco{\tfrac{\pi}{2} ad Y_\Gamma}, \quad K = -{\bf c}_\mk{y}I{\bf c}_\mk{y}^{-1}\mbox{  on } \mc{V}^u\oplus \mk{e},
$$
and define $K = IB$ on $\mk{j}$. Obviously $K$ is regular w.r.to the Cartan subalgebra $\mk{x}^\CC \oplus \mk{o}$.

If we perceive a hypercomplex structure on ${\bf U}$ as a representation of $SU(2)$ on $\mk{u}$, which splits into real 4 dimensionnal irreducible components, then the hypercomplex structures constructed in this section do not depend on $\rho$.
\end{remark}

\subsection{Examples - The simple groups} \label{ex}

In this subsection we present the stems of the irreducible reduced root systems with some comments.

\begin{ex} \label{fraka} The root system $\Delta = \mk{A}_s, \quad \mk{u}_s = su(n+1)$. We have $d = \ksco{\frac{n+1}{2}}$.

We take an orthonormal basis $\{e_0,\dots,e_n\}$ in $\EE^{n+1}$ and put
\begin{gather*}
\Delta^+  =   \{  e_i - e_j\vert i, j = 0, \dots,n , \   i < j \}.
\end{gather*}
The involution $\star$ is given by $\star(e_j) = - e_{n-j},\ j = 0,1,\dots,n$.

The stem is $\Gamma =  \{\gamma\in \Delta^+\vert\star(\gamma) = \gamma\} = \{\gamma_1, \dots , \gamma_d \}$, where
\begin{gather*}\gamma_1 = e_0 - e_n,\dots,\gamma_k = e_{k-1} - e_{n-k +1},\dots,\gamma_d = e_{d-1} - e_{n-d+1},\\
 \Phi^+_{\gamma_k} = \{ e_{k-1} - e_j\vert k \leq j\leq n- k \} \cup  \{ e_j - e_{n-k+1}\vert  k\leq j\leq n-k \}.
\end{gather*}
{\bf Case 1}  $\Delta  = \mk{A}_{2d},\quad \mk{u}_s = su(2d+1)$. We have (see Corollary \ref{star5}):
\begin{gather*}
\mk{o} = \mk{o}_s = span_\CC\{E_k^k + E_{n-k}^{n-k} - 2E_d^d\ \vert \ 0\leq k \leq d-1\},
\end{gather*}
where $E_j^k$ is matrix with 1 on the intersection of the $j$-th column and $k$-th row, 0 elsewhere.
{\bf Case 2}  $\Delta  = \mk{A}_{2d-1},\quad \mk{u}_s = su(2d)$. We have
\begin{gather*}
 \mk{o}_s = span_\CC\{E_k^k - E_{d-1}^{d-1}-E_{d}^d+ E_{n-k}^{n-k} \ \vert \ 0\leq k \leq d-2\}.
\end{gather*}
\end{ex}

\begin{ex}\label{sop2}  The root system $\mk{D}_p, \quad \mk{u}_s = \mk{so}(2p)$.

Let $e_1,\dots,e_p$ be an orthonormal basis of $\EE^p$, we have
\begin{gather*}
 \Delta^+  =  \{  e_i \pm e_j; i, j =1, \dots,p,   \   i< j \}.
\end{gather*}
The basis is $\Pi = \{\alpha_1, \dots, \alpha_p\}$, where
\begin{gather*}
 \alpha_1=e_1-e_2,\dots, \alpha_i=e_i-e_{i+1}, \dots, \alpha_{p-1}=e_{p-1}-e_{p}, \alpha_{p}=e_{p-1}+e_{p}.
\end{gather*}
We have an involution $\nu \in {\bf Aut}_{\Pi}(\Delta)$ given by
\be \label{nuso}   \nu(e_i) = e_i, \ \ 1\leq i\leq p-1 \qquad \nu(e_p) = - e_p.\ee
The stem is $\Gamma = \{\gamma_1,\dots,\gamma_d \},\ d = 2q = 2\ksco{\frac{p}{2}}$, where
\begin{gather*}
\gamma_1 = e_1+e_2,\ \dots,\ \gamma_k=e_{2k-1}+e_{2k}, \dots, \ \gamma_q = e_{2q-1} + e_{2q},\\
\gamma_{q+1} = e_1 - e_{2},    \dots , \gamma_{q+k}=e_{2k-1}-e_{2k}, \dots,  \gamma_{2q} = e_{2q -1} - e_{2 q}.
 \end{gather*}
For $k=1,\dots, q$ we have $\Phi_{\gamma_{q+k}}^+ = \emptyset$ and
\begin{gather*}
\Phi_{\gamma_k}^+ = \{ e_{2k-1} \pm e_j ;  2 k <j \leq p \} \cup  \{ e_{2 k} \pm e_j ;  2 k <j \leq p \}.
\end{gather*}
{\bf Case 1} $p = 2q$. \label{so(4q)}Now $\star$ is trivial, $d = p$. The diagram automorphism $\nu$ is not trivial on $\Gamma$. We have $\quad \nu(\gamma_i) = \gamma_i,  \  i \not\in \{ q, 2 q \}; \quad  \nu(\gamma_q) = \gamma_{2q}$.

The group ${\bf Aut}_{\Pi}(\mk{D}_4)$ is  the permutation group of $\{\alpha_1, \alpha_3, \alpha_4\}$ and leaves  $\alpha_2$ fixed. In this case $ \Gamma = \{\gamma_1, \gamma_2, \gamma_3,\gamma_4\}$, where
\begin{gather*}
\gamma_1 = \alpha_1 + 2 \alpha_2 + \alpha_3 + \alpha_4, \ \  \gamma_2 = \alpha_4, \ \   \gamma_3=\alpha_1, \ \   \gamma_4 = \alpha_3.
\end{gather*}
So ${\bf Aut}_{\Pi}(\mk{D}_4)$  permutes   $\{\gamma_2, \gamma_3, \gamma_4\}$ and leaves $\gamma_1$ fixed.

{\bf Case 2}\label{nuso2} $p = 2q+1$.  Now $\nu = \star$ (see \eqref{nuso}) so  $d = p-1 =2q$ and $\mk{o}_s = \CC h_{e_p}$.
\end{ex}

\begin{ex}\label{tau3} The root system $\Delta = \mk{B}_p,\quad \mk{u}_s = \mk{so}(2p +1)$.

 We take orthonormal elements $e_1,\dots,e_p$ in $\EE^p$ and put
\begin{gather*}
\Delta^+  =  \{  e_i \pm e_j; i, j =1, \dots,p \ and  \  i< j \} \cup \{ e_i ; 1=1,\dots, p\}.
\end{gather*}
If $p = 2q$, then the stem is $ \Gamma = \{\gamma_1, \dots ,  \gamma_{2q}\}$, where
\begin{gather*}
\gamma_1 = e_1 + e_2, \quad \dots , \gamma_k=e_{2k-1}+e_{2k}, \dots, \ \gamma_q = e_{2 q -1} + e_{2q}, \nonumber \\
\gamma_{q+1} = e_1 - e_2, \dots , \gamma_{q+k}=e_{2k-1}-e_{2k}, \dots, \ \gamma_{2q} = e_{2 q -1} - e_{2q}
\end{gather*}
and for $ k=1,\dots, q$ we have $\Phi^+_{q+k} = \emptyset$ and
\begin{gather*}
 \Phi^+_k = \{ e_{2k-1} \pm e_j ;  2 k <j \leq p \} \cup  \{ e_{2 k} \pm e_j ;  2 k <j \leq p \}\cup \{e_{2k-1}, e_{2 k}\}.
\end{gather*}

If $p=2q+1 $ then the stem is $\Gamma = \{\gamma_1, \dots ,  \gamma_{2q+1}\}$, where
\begin{gather*}
\gamma_1 = e_1 + e_2, \gamma_2 = e_3+ e_4,\dots,\gamma_k=e_{2k-1}+e_{2k}, \dots, \ \gamma_q = e_{2q -1} + e_{2q} \\
\gamma_{q+1} = e_1 - e_2,\dots, \gamma_{q+k}=e_{2k-1}-e_{2k},\dots,\gamma_{2q} = e_{2q-1}-e_{2q}, \ \gamma_p = e_p.
\end{gather*}
and for $k=1,\dots,q$ we have $\Phi_{q+k}^+ = \Phi_d^+ = \emptyset$ and
\begin{gather*}
\Phi_k^+ = \{ e_{2k-1} \pm e_j ;  2 k <j \leq p \} \cup  \{ e_{2 k} \pm e_j ;  2 k <j \leq p \}\cup \{e_{2 k-1}, e_{2 k}\}.
\end{gather*}
On the root systems $\mk{B}_p$ the involution $\star$ is trivial, so $d = p$.
\end{ex}

\begin{ex}\label{frakc} The root system $\mk{C}_d,\quad \mk{u}_s = sp(d)$.

 We take an orthonormal basis $e_1,\dots,e_d$ of $\EE^d$ and put
\begin{gather*}
 \Delta^+  =  \{e_i \pm e_j\vert\ i \leq i \leq j\leq d \}.
\end{gather*}

The stem is $\Gamma = \{\gamma_1, \dots ,  \gamma_d\}$, where for $k =1,\dots,d$ we have
\begin{gather*}
\gamma_k = 2e_k,\quad\Phi_{\gamma_k}^+ = \{ e_k \pm e_j\vert\  k <j\leq d\}.
\end{gather*}
\end{ex}

\begin{ex}\label{e7e8} The root systems $\mk{E}_k,\ k = 6,7,8$.

We take an orthonormal basis $\{e_1,\dots,e_8 \}$  of $\EE^{8}$ and denote
\begin{gather}\label{bas1}
 \alpha_1=\frac{1}{2}\left(e_8 + e_1 - \sum_{i=2}^7 e_i\right),\\ \alpha_2 = e_1 + e_2, \ \alpha_3 = e_2-e_1,\ \alpha_4 = e_3 - e_2, \nonumber\\
\alpha_5 = e_4 - e_3,\ \alpha_6 = e_5 - e_4,\ \alpha_7 = e_6 - e_5,\ \alpha_8 = e_7 - e_6.\nonumber
\end{gather}
\begin{figure}[h]\label{dynd}
\setlength{\unitlength}{1mm}
\begin{picture}(130,30)(0,-7)
\linethickness{1pt}
\put(9,0){\circle{2}}
\put(10,0){\line(1,0){10}}
\put(21,0){\circle{2}}
\put(22,0){\line(1,0){10}}
\put(33,0){\circle{2}}
\put(34,0){\line(1,0){10}}
\put(45,0){\circle{2}}
\put(46,0){\line(1,0){10}}
\put(57,0){\circle{2}}
\put(58,0){\line(1,0){10}}
\put(69,0){\circle{2}}
\put(70,0){\line(1,0){10}}
\put(81,0){\circle{2}}
\put(82,0){\line(1,0){10}}
\put(93,0){\circle{2}}
\put(33,1){\line(0,1){10}}
\put(33,12){\circle{2}}
\put(27,10){\mbox{$\alpha_2$}}
\put(8,-4){\mbox{$\alpha_1$}}
\put(20,-4){\mbox{$\alpha_3$}}
\put(32,-4){\mbox{$\alpha_4$}}
\put(44,-4){\mbox{$\alpha_5$}}
\put(56,-4){\mbox{$\alpha_6$}}
\put(68,-4){\mbox{$\alpha_7$}}
\put(80,-4){\mbox{$\alpha_8$}}
\put(92,-4){\mbox{$\gamma_1$}}
\end{picture}
\caption{Extended Dynkin diagram of $\mk{E}_8$.}
\end{figure}
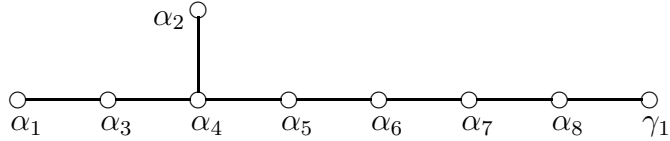

{\bf Case 1} \label{frake8}  The root system $\Delta = \mk{E}_{8}$. We have $d = 8$.

All indices vary in $\{1,\dots,8\}$. Below $\epsilon:\{1,\dots,8\}\lw \{0,1\}$ is any function.
\begin{gather*}
 \Delta^+  =  \{ (e_i \pm e_j)\vert\ i> j \}\cup \left \{\frac{1}{2}\left (e_8  + \sum_{i=1}^7(-1)^{\epsilon(i)} e_i\right )\big| \sum_{i=1}^7 \epsilon(i)\  \mbox{even} \right \}.
\end{gather*}
The basis is $\Pi = \{\alpha_1,\dots, \alpha_8\}$ (see \eqref{bas1}). The highest root is $\gamma_1= e_7+e_8$. Also
\begin{gather*}
\Phi_1^+ =  \{e_j \pm e_i\vert\ i\leq 6 <j \}\cup \left \{ \frac{1}{2}\left (\gamma_1 + \sum_{i=1}^6 (-1)^{\epsilon(i)} e_i \right )\bigl|\sum_{i=1}^6 \epsilon(i) \  \mbox{even} \right \}.\nonumber
 \end{gather*}
Now we move to

{\bf Case 2}\label{frake7}  The root system $\Delta_2 = \mk{E}_7$. We have $d = 7$.

All indices further vary in $\{1,\dots,6\}$, {\bf we denote} $f = e_8 - e_7$ and $\Delta^+_2$ is
\begin{gather*}
 \{f\}\cup\{ (e_i \pm e_j)\vert i> j\}  \cup \left \{\frac{1}{2}\left (f + \sum_{i=1}^6 (-1)^{\epsilon(i)} e_i \right )\big|\sum_{i=1}^6 \epsilon(i) \ \ \mbox{odd} \right\}.
\end{gather*}
The basis $\Pi_2$ of $\Delta_2$ is $\{ \alpha_1, \alpha_2, \dots, \alpha_7\}$ (see \eqref{bas1}).

{\bf The highest root} is $\gamma_2= f = e_8- e_7$. Further we have
\begin{gather*}
\Phi_f^+  = \left \{ \frac{1}{2}\left (f + \sum_{i=1}^6 (-1)^{\epsilon(i)} e_i \right )\big|\ \sum_{i=1}^6 \epsilon(i) \ \ \mbox{odd}\right \};\\
\Delta_3 =  \left \{e_i \pm e_j\vert\ i \neq j \right\} =
\mk{D}_6,\qquad \Delta_3^+ = \left \{ e_i \pm e_j\vert\ i > j\right \}.
\end{gather*}
To go on, move to Example \ref{sop2}, Case 1.

{\bf Case 3}\label{e6}
The root system $\Delta  = \mk{E}_{6}$. Here $d = 4$.

We use the notation from the two preceding examples. All indices vary in $\{1,\dots,5\}$. {\bf We denote} $e =
e_8 - e_7 - e_6 $.
\begin{gather*}
 \Delta^+ = \{e_i \pm e_j \vert\ i > j\} \cup \left \{\frac{1}{2}\left (e + \sum_{i=1}^5(-1)^{\epsilon(i)} e_i\right ) \big|\ \sum_{i=1}^5 \epsilon(i) \ \ \mbox{even}\right \}.
 \end{gather*}

The basis is $\Pi = \{ \alpha_1, \alpha_2,\alpha_3, \alpha_4, \alpha_5, \alpha_6\}$ (see \eqref{bas1}).
The involution $\star$ is given by:
\begin{gather*}
 \star(\alpha_1) = \alpha_6 ,\  \star(\alpha_3) = \alpha_5,\ \star(\alpha_2) = \alpha_2,\ \star(\alpha_4) = \alpha_4,\\
 \mk{o}_s = \CC h_{\alpha_1 - \alpha_6} \oplus \CC h_{\alpha_3 - \alpha_5}.
\end{gather*}
We have $ \quad \gamma_1 = \frac{1}{2}(e_1 + e_2 + e_3 + e_4 +e_5 - e_6 - e_7 + e_8)$ and
\begin{gather*}
\Phi_{\gamma_1}^+  =  \{e_i + e_j\vert\ i < j\}\cup \left \{ \frac{1}{2}\left ( e + \sum_{i=1}^5  (-1)^{\epsilon(i)} e_i \right )\vert\ \sum_{i=1}^5 \epsilon(i)=2\right\};\\
\Delta_2^+ = \left \{ (e_i - e_j)\vert\ i > j\right \}\cup \left\{ \frac{1}{2}\left( e + \sum_{i=1}^5  (-1)^{\epsilon(i)} e_i \right )\big|\ \sum_{i=1}^5 \epsilon(i)=4\right\} .
\end{gather*}
If we denote
\begin{gather*}
f_1 = \frac{1}{2}\left (\sum_{i=1}^7 e_i- e_8\right),\  f_2=e_1,\ f_3=e_2,\  f_4=e_3,\ f_5=e_4,\ f_6=e_5,
\end{gather*}
then we may represent $\Delta_2 = \{f_i-f_j\vert 1\leq i,j \leq 6\} = \mk{A}_5$, whence we go to Example \ref{fraka}.

Thus we get $\Gamma = \{ \gamma_1, \dots,\gamma_4 \}$, where $\gamma_2 = f_1 - f_6,\ \gamma_3 = f_2 - f_5,\ \gamma_4 = f_3 - f_4$.
\end{ex}

\begin{ex}\label{f4} The root system $\Delta = \mk{F}_{4}$.

Let $\{e_1,e_2,e_3,e_4\}$ be an orthonormal basis of $\EE^4$. We have
\begin{gather*}
\Delta^+ = \{ e_i \pm e_j\vert\ i< j\} \cup \{ e_i; 1 \leq i \leq 4\}\cup\left \{\frac{1}{2}( e_1 \pm e_2 \pm e_3 \pm e_4)\right \}.
\end{gather*}
The basis is $\Pi = \{\alpha_1, \alpha_2, \alpha_3, \alpha_4\}$, where
\begin{gather*}
\alpha_1 = e_2 - e_3,\ \alpha_2 = e_3-e_4,\ \alpha_3 = e_4,\ \alpha_4 = \frac{1}{2}(e_1 - e_2 - e_3 - e_4).
\end{gather*}

The highest root is $\gamma_1=e_1 + e_2$ and
\begin{gather*}
\Phi_{\gamma_1}^+  = \{e_1 \pm e_i, e_2 \pm e_i\vert\ i=3, 4 \} \cup \{e_1, e_2\} \cup  \left \{ \frac{1}{2}\left (e_1 + e_2 \pm e_3 \pm e_4 \right ) \right \};\\
\Delta_2^+ = \{e_1 - e_2, e_3, e_4\} \cup \{ e_3 \pm e_4\}\cup \left \{ \frac{1}{2}\left ( e_1 - e_2 \pm e_3 \pm e_4 \right ) \right \}.
\end{gather*}
If we denote $ f_1=\frac{1}{2}( e_1-e_2),\ f_2 = \frac{1}{2}( e_3+e_4),\ f_3=\frac{1}{2}( e_3-e_4)$, then
\begin{gather*}
\Delta_2 =  \{f_i \pm f_j\vert 1\leq i,j \leq 3\}\setminus \{0\} =
\mk{C}_3.
\end{gather*}
From Example \ref{frakc} we get $\Gamma = \{e_1+e_2,\ e_1-e_2,\ e_3 +e_4,\ e_3 - e_4 \}$.
\end{ex}

\begin{ex}\label{frakg2}The root system $\Delta = \mk{G}_2$.

Let $\{e_1,e_2,e_3\}$ be an orthonormal basis of $\EE^3$.  We have
\begin{gather*}
\Delta = \{\pm(e_i - e_j),\pm(2e_k - e_i - e_j)  \vert \ i,j,k = 1,2,3;\ i\neq j, i\neq k, j\neq k\};\\
\Pi = \{\alpha,\beta\}, \quad \alpha = e_1 - e_2,\ \beta = e_2 + e_3 - 2e_1,\\
\Delta^+ = \{\alpha, \beta, \beta+\alpha,\beta+2\alpha,\beta+3\alpha,2\beta+3\alpha\}.
\end{gather*}
The highest root is $\gamma = 3\alpha + 2\beta = 2e_3 - e_1 - e_2,\quad \Phi_\gamma^+ = \{\beta, \beta+\alpha,\beta+2\alpha,\beta+3\alpha\}$. The stem is $\Gamma = \{\gamma,\alpha\}$.
\end{ex}

\section{The hypercomplex structures}\label{hs}

In this section we prove that up to equivalence, the hypercomplex structures described in the preceding section are all the hypercomplex structures on $\mk{u}$.
So we assume that $I$ is any admissible $\mk{b}^+$ complex structure on $\mk{u}$ and $J$ is an {\bf integrable} complex structures on $\mk{u}$ matching $I$.

In this section we use freely the conventions and notations of sections \ref{css} and \ref{hss}. In particular we use the direct decompositions
\begin{gather}\label{ca3}
\mk{g} = \mk{m}_I^+ \oplus \mk{m}_I^- = \mk{n}^+ \oplus \mk{h}^+ \oplus \mk{n}^- \oplus \mk{h}^- = \mk{h}\bigoplus_{\alpha\in\Delta}\mk{g}(\alpha).
\end{gather}
When $\mk{a}$ is a direct summand in one of these decompositions and we write $pr_\mk{a}:\mk{g}\lw\mk{a}$ we always mean projection along the complementary component in the above formula. Obviously the basis of Definition \ref{bo2} is well adapted to such practices.

We work with the "complexified" Nijenhuis tensor, i.e. we extend $N(X,Y)$ to $\mk{g}$ by complex linearity.

\subsection{The Nijenhuis tensor} \label{sus2}

\begin{prop}\label{ca1}
Let $\alpha,\beta \in \Delta^+,\quad q = 1,\dots,m$. Then
\begin{gather}\label{ca6}
a_{\beta,\alpha}(\alpha + \beta)(U_q) = pr_{\mk{g}(-\beta)}\sco{\sum_{\nu\in\Delta^+}\eta_{\nu,q}[E_{-\nu},E_\alpha]}
\end{gather}
\end{prop}
\bpr
We decompose the element $JN_J(U_q,E_\alpha)\in \mk{g}$ in the basis of Definition \ref{bo2} using formula \eqref{bo8}. Our purpose is to compute the coefficient before $E_{-\beta}$. From integrability of $J$ we have
\begin{gather*}
0 = JN_J(U_q,E_\alpha) =J[JU_q, JE_\alpha] - J[U_q,E_\alpha]   +  [U_q,JE_\alpha] + [JU_q,E_\alpha] =  \\
A - \alpha(U_q)\sum_{\beta\in \Delta^+}a_{\beta,\alpha} E_{-\beta}  - \sum_{\beta\in \Delta^+} a_{\beta,\alpha} \beta(U_q)E_{-\beta}  + \sum_{\nu\in\Delta^+}\eta_{\nu,q}[E_{-\nu},E_\alpha],
\end{gather*}
where $ A = J[JU_q, JE_\alpha] - \alpha(U_q)pr_{\mk{h}}(JE_\alpha) + \alpha(pr_{\mk{h}}(JU_q))E_\alpha\in \mk{b}^+$. For $\beta\in \Delta^+$ the coefficient at $E_{-\beta}$ must vanish, whence the proposition.
\epr

\begin{coro}\label{cor42} If $\alpha, \beta \in \Delta^+,\ \alpha+\beta\not\in \Delta$, then
$a_{\alpha,\beta} = a_{\beta,\alpha} = 0$.
\end{coro}
\begin{proof}

By the assumption, in formula \eqref{ca6} the RHS is 0. On the other hand, the functional $\alpha+\beta$ is real at $\mk{h}_\RR$ and nonzero, so it is obvious that we may choose such a $q$ that $(\alpha+\beta)(U_q) \neq 0$.
\end{proof}

\begin{coro}\label{cu1} If $\gamma$ is a maximal root, then $JE_\gamma \in \mk{h}^-$.
\end{coro}
\begin{proof}
By Corollary \ref{cor42}, for all $\nu \in \Delta^+$ we have $a_{\nu,\gamma} = 0$.
\end{proof}
Further,  \eqref{ca6} obviously implies
\begin{coro}\label{su6}
Let $\alpha,\beta \in \Delta^+$ and  $ q=1,\dots, m$. If  $\gamma = \alpha
+ \beta \in \Delta$ then
$$
a_{\beta,\alpha}\gamma(U_q) = - N_{\gamma,-\alpha}\eta_{\gamma,q}.
$$
\end{coro}

\begin{coro} \label{hee4}
If $\alpha, \beta, \gamma \in \Delta^+$ and $\gamma = \alpha + \beta$, then
$ N_{\gamma,-\alpha}a_{\alpha,\beta} = N_{\gamma,-\beta}a_{\beta,\alpha}.$
\end{coro}
\begin{proof} Under the condition obviously
$N_{\gamma,-\beta} \neq 0 \neq N_{\gamma,-\alpha}$. We choose a q so that $\gamma(U_q)\neq 0$ and apply twice the formula in Corollary \ref{su6}.
\end{proof}

\begin{coro} \label{ss7}
 For any $\alpha,\beta \in \Delta^+$ we have $a_{\alpha,\beta}= 0 \iff a_{\beta,\alpha} = 0$.
\end{coro}
\begin{proof} If $\alpha + \beta \not \in \Delta^+$ then we use Corollary \ref{cor42}, otherwise - Corollary \ref{hee4}.
\end{proof}
Before going on with the Nijenhuis tensor we introduce some convenient notation.
Let $\gamma \in \Delta^+, \quad JE_\gamma \in \mk{h}$. We denote
\begin{gather}\label{ca7}
V_\gamma = JE_\gamma \in \mk{h}^-,\qquad U_\gamma = JE_{-\gamma}=-\tau (JE_\gamma) = - \tau (V_\gamma)\in \mk{h}^+,
\end{gather}
From the above definition and $\alpha(\tau(H)) = -\ol{\alpha(H)}$ we get
\begin{gather}\label{ca5}
\gamma(V_\gamma) =  \ol{\gamma(U_\gamma)}.
\end{gather}
Now we may compute
\begin{prop}\label{eee4}
 Let $I$ be an admissible complex structure and let $J$ match $I$. Let $\gamma \in \Delta^+,\quad JE_\gamma \in \mk{h}$. Then
\begin{gather} \label{ca11}
|\gamma(U_\gamma)| = 1,\quad \gamma(IH_\gamma) = 0;\\
JE_\gamma = \frac{1}{2}\gamma(V_\gamma)\sco{H_\gamma + \ri IH_\gamma},\quad JE_{-\gamma} = \frac{1}{2}\gamma(U_\gamma)\sco{H_\gamma - \ri IH_\gamma}.\label{ca10}
\end{gather}
\end{prop}
\bpr Integrability gives
\begin{gather*}
N_J(E_\gamma, E_{-\gamma}) = [V_\gamma, U_\gamma] - [E_\gamma,E_{-\gamma}]- J[V_\gamma, E_{-\gamma}]- J[E_\gamma,U_\gamma]\\
 = -H_\gamma + \gamma(V_\gamma)U_\gamma + \gamma(U_\gamma)V_\gamma = 0.
\end{gather*}
For the present computation we denote $a = \gamma(U_\gamma)$.
Now we apply $I$ on the last expression to get the second equation of the following system
\begin{gather}\label{we2}
H_\gamma = \ol{a}U_\gamma + aV_\gamma,\quad
IH_\gamma =  \ri\ol{a}U_\gamma - \ri aV_\gamma.
\end{gather}
First we use \eqref{we2} to compute
\begin{gather*}
0 = N_J(H_\gamma,IH_\gamma) = -[aE_\gamma + \ol{a}E_{-\gamma},\ri aE_\gamma - \ri \ol{a}E_{-\gamma}] \\
- [H_\gamma,IH_\gamma] - J[H_\gamma,\ri aE_\gamma - \ri \ol{a}E_{-\gamma}] + J[aE_\gamma + \ol{a}E_{-\gamma},IH_\gamma]\\ =  2\ri(|a|^2-1)H_\gamma  - \ri\gamma(IH_\gamma)IH_\gamma.
\end{gather*}
Because $H_\gamma, IH_\gamma$ are linearly independent, integrability implies \eqref{ca11}.
Now using \eqref{ca11} we solve the system \eqref{we2} to get \eqref{ca10}.
\epr
\begin{remark} At first glance formula \eqref{ca10} contains something like a vicious circle - we determine $V_\gamma = JE_\gamma$ using a circle parameter $\gamma(V_\gamma)$ on the RHS.

As we shall prove further $JE_\gamma \in \mk{h}$ iff $\gamma \in \Gamma$ (see Theorem \ref{ns1}).
Actually, given an admissible $\mk{b}^+$ complex structure $I$ and a matching $J$, we have proved that
\begin{gather}\label{ca12}
JE_\gamma = V_\gamma = - \ri\gamma(V_\gamma)Q_\gamma,\quad JE_{-\gamma} = U_\gamma = -\ri\gamma(U_\gamma)P_\gamma.
\end{gather}
(See Definition \ref{mj1} for $P_\gamma,Q_\gamma$). The important point here is that any matching complex structure $J$ sends the stem nilpotent $E_\gamma\in \mk{f}^+$ to $Q_\gamma \in \mk{h}^+$ multiplied by a complex number of norm 1, thus we recover the parameters $\rho_\gamma$ from Section \ref{ste}. We use this further to identify the Cayley transform which produces $J$ - see Definition \ref{qqq8} and further.
\end{remark}
 \begin{prop}\label{su1}
 Let $\alpha,\beta,\gamma \in \Delta^+,\ JE_\gamma\in\mk{h}$. Then
\begin{gather*}
a_{\beta,\alpha}(\beta + \alpha)(U_\gamma) = pr_{\mk{g}(-\beta)}([E_\alpha,E_{-\gamma}]).
 \end{gather*}
\end{prop}
\bpr By integrability of J we have
\begin{gather*}
0 = N(E_{-\gamma}, E_\alpha) =[U_\gamma, JE_\alpha]-[E_{-\gamma},E_\alpha] - J[U_\gamma, E_\alpha]-J[E_{-\gamma},JE_\alpha]\\
 =[E_\alpha,E_{-\gamma}]  - \sum_{\beta\in \Delta^+} \beta(U_\gamma)a_{\beta,\alpha} E_{-\beta} -\alpha(U_\gamma)\sum_{\beta\in \Delta^+} a_{\beta,\alpha} E_{-\beta} +A,
\end{gather*}
where $A \in \mk{b}^+$. The statement of the proposition comes from equating to zero the coefficient at $ E_{-\beta}$ in the last expression.
\epr

\begin{coro}\label{le1} Let $\alpha, \beta ,\gamma\in \Delta^+,\ \alpha + \beta = \gamma,\quad JE_\gamma \in \mk{h}$. Then
\begin{gather*}
a_{\beta,\alpha} = N_{\gamma,-\alpha}\gamma(V_\gamma) \neq 0.
\end{gather*}
\end{coro}
\begin{proof} Follows trivially from Proposition \ref{su1}, by \eqref{ca11} and the obvious fact that $N_{\gamma,-\alpha} \neq 0$.
\end{proof}

\begin{lemma}\label{hee1}   Let $\alpha,\beta \in \Delta^+,\quad  \gamma \in \Gamma, \quad JE_\gamma \in \mk{h}$. If $\alpha + \beta \neq \gamma$, then
$$
\sum_{\nu \in \Phi_\gamma^+ } a_{\nu,\beta}a_{\alpha,\mu(\nu)}N_{\gamma,-\nu } = 0.
$$
\end{lemma}
\bpr
In the follwing computation we keep explicit only terms with component in $\mk{n}^-$ . We have
\begin{gather*}
0 = N(E_\gamma, E_\beta) =[V_\gamma, JE_\beta] - J[V_\gamma, E_\beta]-J[E_\gamma,JE_\beta] + A \nonumber\\
\sum_{\alpha\in\Delta^+}a_{\alpha,\beta} (\beta + \alpha)(V_\gamma)E_{-\alpha} - J\left[E_\gamma, pr_\mk{h}(JE_\beta) + \sum_{\nu\in\Delta^+}a_{\nu,\beta} E_{-\nu}\right] +A\\
=   -\sum_{\alpha\in\Delta^+}a_{\alpha,\beta} (\beta + \alpha)(V_\gamma)E_{-\alpha}
- \sum_{\nu\in\Delta^+}a_{\nu,\beta}J[E_\gamma,E_{-\nu}] +B\\
=-\sum_{\alpha\in \Delta^+} \left(a_{\alpha,\beta}{(\alpha+\beta )(V_\gamma)} + \sum_{\nu \in \Phi_\gamma^+ }N_{\gamma,-\nu } a_{\nu,\beta} a_{\alpha,\mu(\nu)}\right ) E_{-\alpha} +C.
\end{gather*}
From integrability we conclude that for each $\beta,\alpha,\gamma$ as assumed, we have
\begin{gather*}
a_{\alpha,\beta}(\alpha+\beta)(V_\gamma) + \sum_{\nu \in \Phi_\gamma^+ } N_{\gamma,-\nu} a_{\nu,\beta}a_{\alpha,\mu(\nu)}=0.
\end{gather*}
When $\beta +\alpha \neq \gamma$, Proposition \ref{su1} gives\footnote{Because $(\alpha+\beta )(V_\gamma) = \ol{(\alpha+\beta )(U_\gamma)}$.} $a_{\alpha,\beta}(\alpha+\beta)(V_\gamma) = 0$, whence the lemma.
  \end{proof}

\subsection{The coefficients of {\bf a}}\label{ca}

We recall that $\Gamma$ is the stem of $\Delta^+$.

\begin{prop} \label{hee31}
 Let $\gamma \in \Gamma,\ J(E_\gamma) \in \mk{h}$ and $\alpha \in \Phi_\gamma^+,\quad \beta \in \Delta^+$ Then
 $a_{\beta, \alpha} \neq 0$ if and only if $\alpha + \beta = \gamma$.
  \end{prop}
  \begin{proof}
If $\beta \in \Phi_\gamma^+ $, then from Corollary \ref{eu5}, a), we know  that $\alpha + \beta \in \Delta$ iff $\alpha +\beta = \gamma$, so Corollaries \ref{cor42} and \ref{le1} give
\be \label{hee2}
 a_{\alpha, \beta} \neq 0 \iff  \alpha + \beta = \gamma,\quad \alpha, \beta   \in \Phi_\gamma^+.
 \ee

  Let  $\beta \not \in \Phi_\gamma^+ $. If $\beta + \alpha \not \in \Delta^+$, then $a_{\alpha,\beta}=0$ by Corollary \ref{cor42}. So we have to treat just the case $\gamma \neq \beta + \alpha  \in \Delta^+$ as in  Lemma \ref{hee1}.
 Now if $\alpha,\nu,  \mu(\nu) \in \Phi_\gamma^+ $, then by
 \eqref{hee2} $a_{\alpha,\mu(\nu)} \neq 0$ if and only if  $\nu = \alpha$. Thus, the equality from Lemma \ref{hee1} reduces to $a_{\beta,\alpha} N_{\gamma ,-\alpha}  =0$.

We have $\alpha \in \Phi_\gamma^+ $, hence  $N_{\gamma ,-\alpha }
\neq 0$, whence $a_{\alpha,\beta} = 0$. The proposition is proved.
\end{proof}

Now we can prove
\begin{theorem} \label{ns1} Let I be an admissible $\mk{b}^+$ complex structure on $\mk{u}$ and let $J$ match I. Then

a) If $\gamma \in \Gamma$, then $J(E_\gamma) \in \mk{h}^-$;

b) If $\alpha \in \Phi^+,\ \beta\in\Delta^+$, then $ a_{\alpha, \beta} \neq 0$ if and only if  $\beta = -s_\gamma(\alpha)$.

\end{theorem}
\begin{proof}Let $\Gamma = \{\gamma_1, \dots, \gamma_d\}$ be the stem of $\Delta^+$.
We know that $\gamma_1$ is a maximal root, so by Corollary \ref{cu1} we have $JE_{\gamma_1} \in \mk{h}$.

Now by Proposition \ref{hee31} we conclude that for any $ \alpha \in \Phi_{\gamma_1}^+,\ \beta \in \Delta^+$ we have
$$
a_{\alpha ,\beta} \neq 0  \iff \alpha + \beta = \gamma_1.
$$
Now we assume that for some $k < d$ we have $J(E_{\gamma_i}) \in \mk{h},\quad i=1,\dots,k$ and
\begin{gather}\label{ia1}
a_{\alpha ,\beta} \neq 0  \iff  \beta = \mu(\alpha) ,\quad \alpha \in \Phi_{\gamma_1}^+ \cup\dots \cup\Phi_{\gamma_k}^+,\ \beta \in \Delta^+ .
\end{gather}
If $\alpha \in \Phi_{\gamma_1}^+\cup\dots\cup\Phi_{\gamma_k}^+$ then by the definition of the stem $\gamma_{k+1}\neq \mu(\alpha)$, hence by the  induction assumption \eqref{ia1} and Corollary \ref{ss7} we have
$a_{\gamma_{k + 1}, \alpha}= a_{\alpha, \gamma_{k + 1}} = 0$.

If $\alpha \in \Phi_{\gamma_{k+1}}^+\cup\dots\cup\Phi_{\gamma_d}^+ \cup \Gamma$, then by Corollary \ref{yyy7} we have $\alpha + \gamma_{k+1} \not\in\Delta $, whence by Corollary \ref{cor42} we conclude $a_{\gamma_{k + 1},\alpha} = a_{\alpha, \gamma_{k + 1}} = 0$.

Thus for each $\alpha \in \Delta^+$ we nave $a_{\gamma_{k + 1},\alpha} = a_{\alpha, \gamma_{k + 1}} = 0$, which means
\begin{gather}\label{qqq1}
JE_{\gamma_{k+1}} \in \mk{h}.
\end{gather}
Now let $\alpha \in \Phi_{\gamma_{k+1}}^+$

If $\beta = \mu(\alpha)$ then Corollary \ref{le1} and \eqref{qqq1} give $a_{\alpha,\beta} \neq 0$.

If $\beta \in \Delta^+$ and $\beta \neq \mu(\alpha)$, we apply Proposition \ref{hee31} so
$$
 a_{\alpha,\beta} \neq 0 \iff  \alpha + \beta  = \gamma_{k+1},\quad \alpha \in \Phi_{\gamma_{k+1}}^+,\ \beta \in \Delta^+,
$$
which combined with the assumption \eqref{ia1} gives
$$
  a_{\alpha,\beta} \neq 0 \iff \mu(\alpha)= \beta,\quad \alpha \in \bigcup_{i=1}^{k+1}\Phi_{\gamma_i}^+, \  \beta \in \Delta^+.
$$
Our induction is complete, the theorem is proved.
\end{proof}

\begin{coro}\label{qqq32} The matrix $ {\bf a}$ is antisymmetric.
\end{coro}
\bpr
From Theorem \ref{ns1} we know that $a_{\alpha,\beta} \neq 0$ iff $\alpha \in \Phi^+$ and $\beta = \mu(\alpha)$. The result follows from  Proposition \ref{www1} and Corollary \ref{hee4}.
\epr

\begin{coro}\label{qqq2} $ J(\mk{f}^+) \subset \mk{h}^-.$
\end{coro}
\bpr
Follows directly from Theorem \ref{ns1} and $J(\mk{m}^+_I) = \mk{m}^-_I$.
\epr

\begin{coro}\label{qqq3}
If $rank(\mk{u}) < 2d$, then ${\bf U}$ carries no hypercomplex structure.
\end{coro}
\bpr
Follows from Corollary \ref{qqq2}, the fact that $J$ is bijective and $2dim(\mk{h}^-) = rank(\mk{g})$.
\epr

 From here on, we assume (often implicitly) that $rank(\mk{g}) \geq 2d$.
\begin{coro}
A semisimple compact Lie group ${\bf U}$ carries a hypercomplex structure if and only if
$$ {\bf U} \cong SU(2d_1+1)\times\dots\times SU(2d_n+1),\quad d_1,\dots,d_n\in \NN.$$
\end{coro}
\bpr
The only simple group with $rank(\mk{g}) = 2d$ is $SL(2n+1,\CC)$ (see subsection \ref{ex}). On the other hand, existence of a hypercomplex structure for our ${\bf U}$ follows from Remark \ref{ij31}.
\epr

Now we are ready to determine the complex structure $J$ on
$\mc{V}$ (see Definition \ref{di1}).
\begin{prop} \label{mxj1}  Let $I$ be an admissible complex structure and let $J$ match $I$. If $\gamma \in \Gamma,\quad \alpha \in \Phi_{\gamma}^+$, then
\begin{gather} \label{mxj11}
JE_\alpha = N_{\gamma,- \alpha}\gamma(V_\gamma)E_{s_\gamma(\alpha)}.
 \end{gather}
\end{prop}
\bpr  We denote $\beta = \mu(\alpha) = -s_\gamma(\alpha)$. From Theorem \ref{ns1} we have
\begin{gather*}
JE_\alpha =a_{\beta,\alpha}E_{-\beta} + H, \qquad H \in \mk{h}^-.
\end{gather*}
In the following computation we keep explicit only terms which nave nontrivial projection to $\mk{h}$.
 \begin{gather*} N_J(E_\gamma,E_\alpha) = [V_\gamma,a_{\beta,\alpha}E_{-\beta} + H]  - J[E_\gamma,a_{\beta,\alpha}E_{-\beta} + H]- J[V_\gamma,E_\alpha] \\
 = \gamma(H)V_\gamma  - N_{\gamma,-\beta}a_{\beta,\alpha}JE_\alpha - \alpha(V_\gamma)JE_\alpha +A\\
 = \gamma(H)V_\gamma - ( N_{\gamma,-\beta}a_{\beta,\alpha} + \alpha(V_\gamma))H +B,
 \end{gather*}
 where $A,B \in \mk{n}^-$. Now we use $a_{\beta,\alpha} = N_{\gamma,-\alpha}\gamma(V_\gamma)$
(see  Corollary \ref{le1}) and $\quad N_{\gamma, -\alpha}N_{\gamma, -\beta}= -1$ (Proposition \ref{www1}) to get
\begin{gather}\label{mxj4}
\gamma(H)V_\gamma + \beta(V_\gamma)H = 0
\end{gather}
We apply $\gamma$ to this equation and obtain  $\gamma(H)(\gamma + \beta)(V_\gamma) = 0 $. By   Proposition \ref{eee4} and $\beta\in\Phi^+_\gamma$ we have
   $$
\beta(V_\gamma) = \gamma(V_\gamma)\sco{\frac{1}{2} + \mbox{ imaginary number }} \neq 0.$$
Thus $(\gamma + \beta)(V_\gamma) \neq 0 $, whence $\gamma(H) = 0$.
Now by \eqref{mxj4} and $\beta(V_\gamma)\neq 0$ we get $H = 0$. The Proposition follows.
\epr
Formula \eqref{mxj11} obviously implies
\begin{coro} \label{eee11}  Let $I$ be an admissible $\mk{b}^+$ complex structure and let $J_1,J_2$ be two complex structures matching $I$.

If $J_1E_\gamma = J_2E_\gamma$ for each $\gamma \in \Gamma$, then $J_1E_\alpha = J_2E_\alpha$ for each $\alpha \in \Delta$.
\end{coro}

Given an admissible $\mk{b}^+$ complex structure $I$, Proposition \ref{mxj1} determines the action of a matching $J$ on the invariant\footnote{Of course one point of Proposition \ref{mxj1} is proving the $J$ invariance of $\mc{V}$.} subspace $\mc{V}$. The result is so clean that it gives us more precise description of the matching Cayley structure $I_{\bf c}$ than we achieved in subsection \ref{sus1}.

\subsection{The action of J on the extended stem subalgebra}
Now we return to the notations of Section \ref{ste}. We show that if $I$ is any admissible complex structure  on $\mk{u}$ and if $J$ matches $I$, then $J = I_{\bf c}$ (as in Section \ref{ste}) for a certain value of the torus parameter $\rho$, namely:
\begin{df}\label{qqq8} Let $\gamma \in \Gamma$. We denote:
\begin{gather*}
\rho_\gamma = \ri\gamma(JE_{-\gamma}),\quad\rho =\{\rho_\gamma\vert\gamma\in\Gamma\} \\
 \end{gather*}
\end{df}

The first equality in formula \eqref{ca11} gives $|\rho_\gamma|=1$ whence we may use all the entities from  Definitions \ref{we1}, \ref{eee81}.

In particular for any $\gamma\in\Gamma$ from \eqref{ca10} we get:
\begin{gather}\label{oo51}
 JE_\gamma = V_\gamma = \ol{\rho_\gamma}(W_\gamma
 + \ri Z_\gamma),\quad
   JE_{-\gamma} = U_\gamma = - \rho_\gamma (W_\gamma
 - \ri Z_\gamma).
 \end{gather}

\begin{prop} \label{eee1}  Let $I$ be an admissible complex structure and let $J$ match $I$. Then for any $\gamma \in \Gamma$ we have
\begin{gather}\label{eee3}
IX_\gamma =  Y_\gamma,\quad IW_\gamma = Z_\gamma;\qquad JX_\gamma =  W_\gamma ,\quad JZ_\gamma = Y_\gamma.
\end{gather}
\end{prop}
\begin{proof} The first and second equality in \eqref{eee3} come from the definition of a $\mk{b}^+$ complex structure $I$. The third and fourth equality come by solving the system \eqref{oo51} for $W_\gamma , Z_\gamma $.
 \end{proof}

The conclusions of Proposition \ref{eee1} and Corollary \ref{oo2}
coincide, but the assumptions are different. The coincidence means
of course, that if $J$ is an arbitrary complex structure matching
$I$, then $J = I_{\bf c}$  on the extended stem subalgebra
$\mk{e}$ (with $\rho$ as in Definition \ref{qqq8}, see also Remark
\ref{di7}). Combining Corollary \ref{eee11} with Proposition
\ref{eee1} we get

\begin{coro} \label{eee21}  Let $I$ be an admissible $\mk{b}^+$ complex structure and let $J$ match $I$. Then
each of the subspaces $\mk{e}, \ \mc{V}$ is $J$-stable. Moreover, for each $X \in \mk{e} \oplus \mc{V}$ we have $JX = {\bf c}I{\bf c}^{-1} X$, where the Cayley transform ${\bf c} = {\bf c}[\rho]$ is as in Definition \ref{xi_p} and $\rho_\gamma = \ri\gamma(JE_{-\gamma})$ for each $\gamma\in\Gamma$.
\end{coro}

Corollary \ref{oo2} was proved under the assumption that $\mk{z} \subset \mk{o}_u$, which is also the sufficient condition of our general existence Theorem \ref{suf1}.
We prove next that the condition  $\mk{z} \subset \mk{o}_u$ is also necessary\footnote{Obviously we could now deduce the condition $J\mk{w}\subset \mk{o}_u$ from Proposition \ref{la1}, and the fact that $J = I_{\bf c}$ on $\mk{e}$. We believe that the following direct proof from integrability is more beautiful.}  for admissibility of $I$. We begin with:
\begin{prop} \label{cor5}  Let $I$ be an admissible complex structure and let $J$ match $I$. If $\gamma,\delta \in \Gamma,\ \gamma\neq \delta$, then $\gamma(JE_\delta) = 0$.
\end{prop}
\begin{proof}
Using Theorem \ref{ns1} and strong orthogonality of $\Gamma$  we compute $JN_J (E_\gamma,E_\delta)$ to get:
\begin{gather*}
0 = [JE_\gamma,E_\delta] + [E_\gamma, JE_\delta]
= \delta(V_\gamma)E_\delta -\gamma(V_\delta)E_{\gamma},
\end{gather*}
hence $\gamma(JE_\delta) = \delta(JE_\gamma) = 0$.
\end{proof}

\begin{coro}\label{qqq77} Let I be an admissible $\mk{b}^+$ complex structure on $\mk{u}$. Let $\Gamma$ be the stem of $\Delta^+$.
If $\gamma,\delta \in \Gamma$, then $\gamma(Z_\delta) = 0$.
\end{coro}
\bpr Let $J$ be any complex structure matching $I$. From
Proposition \ref{eee4} we know that  $\gamma(IH_\gamma) = 0$. If
$\gamma\neq \delta$, then by \eqref{oo51}, Proposition \ref{cor5}
and strong orthogonality of $\Gamma$ we have  $0 =
\gamma(JE_\delta) = \ol{\rho_\delta}(\gamma(W_\delta) -
\ri\gamma(Z_\delta)) = \ri\ol{\rho_\delta}  \gamma(Z_\delta)$. The
corollary is proved. \epr We have a useful consequence of
Proposition \ref{qqq77}
\begin{coro}\label{qqq47}
If $I$ is admissible and $\Gamma = \{\gamma_1,\dots,\gamma_d\}$, then $\gamma_j(P_k) = \gamma_j(Q_k) = \ri \delta_{jk}$.
\end{coro}
Formula \eqref{ca12} and Proposition \ref{mj5} give
\begin{prop}\label{qqq7} Let $J$ match $I$, then $J(\mk{f}^+) = \mk{v}^-,\quad J(\mk{f}^-) =  \mk{v}^+$.
\end{prop}

If we add Theorem \ref{suf1} to Corollary \ref{qqq77} we obtain our solution of {\bf Problem A}:
\begin{theorem} \label{mai1}
Let $\mk{u}$ be a compact Lie algebra and let $I$ be a $\mk{b}^+$ complex structure on $\mk{u}$. Let $\Gamma$ be the stem of $\Delta^+$. Then $I$ is admissible if and only if $dim(\mk{u})$ is divisible by 4 and for each $\gamma, \delta\in \Gamma$ we have $\gamma(IW_\delta) = 0$.
\end{theorem}
We have chosen to express the necessary and sufficient condition
for admissibility of $I$ in the most classical terms,  only using
the notion of stem.
\begin{coro}\label{suf2}
A compact Lie group ${\bf U}$ carries a hypercomplex structure if and only if $rank(u) = 2d + 4k$, where $d$ is the number of elements in the stem $\Gamma$ and $k$ is a nonnegative integer.
\end{coro}

\subsection{The nearest to semisimple}
In the previous subsection we solved {\bf Problem A} from the introduction of this paper. Now we proceed to {\bf Problem B}, that is, we assume that $I$ is an admissible $\mk{b}^+$ complex structure on $\mk{u}$ and describe all complex structures $J$ matching $I$.

By Proposition \ref{eee1} we know that on the extended stem subalgebra $\mk{f}_u \oplus \mk{z}$, any $J$ matching $I$ coincides with the structure $I_{\bf c}$ for some $\rho$ (see Definition \ref{qw4}) .

So we go on to determine the remaining coefficients of the matrix of $J$ (see Definition \ref{bo2}).
Now that we assume $\mk{z}\subset \mk{o}_u$, we shall use the notations of Definition \ref{mj1}.

We assume $rank(\mk{g}) = 2d + 2p $, where $p$ is a nonnegative even integer,
\begin{gather*}
\mk{m}_I^+ = \mk{v}^+ \oplus  \mk{j}^+ \oplus  \mk{f}^+  \oplus \mc{V}^+, \quad dim(\mk{v}^+) = dim(\mk{f}^+) = d,\ dim(\mk{j}^+) = p.
\end{gather*}

In the first place we have the vectors $P_\gamma = W_\gamma - \ri Z_\gamma,\ Q_\gamma = W_\gamma + \ri Z_\gamma$, which are a basis for the subspaces $\mk{v}^+,\mk{v}^-$ respectively.

The following theorem improves the constructive Proposition \ref{ij1} in particular.

\begin{theorem}\label{qqq4} Let $2d = rank(\mk{g})$ and let $I$ be an admissible $\mk{b}^+$ complex structure. Then there is exactly one (up to choice of $\rho$) matching complex structure  $J$. For $\gamma \in \Gamma$ and $\alpha\in \Phi^+_\gamma$  the operator $J$ (with $\rho_\gamma = 1$ for each $\gamma\in\Gamma$) is given by:
\begin{gather*}
 JW_\gamma = - X_\gamma,\ JZ_\gamma = Y_\gamma;\quad    JX_\alpha =- N_{\gamma,-\mu(\alpha)}Y_{\mu(\alpha)},
\end{gather*}
where $X_\alpha = \tfrac{1}{2}(E_\alpha - E_{-\alpha}),\ Y_\alpha = \tfrac{\ri}{2}(E_\alpha + E_{-\alpha})$.
\end{theorem}
\begin{proof}
The first and the second equalities come from Proposition \ref{eee1}. The third is the result of Proposition \ref{mxj1}.
\end{proof}
We give several equivalent forms of the admissibility condition. Recall that in Subsection \ref{di} we introduced and studied a representative of the opposition involution $\phi = \exp(\pi adX_\Gamma)$.
\begin{coro}\label{min1}
Let $2d = rank(\mk{g})$, let $I$ be a $\mk{b}^+$ complex structure on $\mk{u}$. Then $I$ is admissible if and only if any of the following three equivalent conditions holds

a) $\phi(\mk{m}^+_I) = \mk{m}^-_I$.

b) $\phi(\mk{h}^+) = \mk{h}^-$.

c) $\phi\circ I = -I\circ \phi$.
\end{coro}
\bpr
In any case $\phi(\mk{n}^+) = \mk{n}^-$, so a) is equivalent to b).

It is trivial that  c) is equivalent to b) (imitate the proof of Proposition \ref{oo1}).

Now assume that $I$ is admissible. Then for any $W\in \mk{w}$ we have $IW\in \mk{o}$, whence by Corollary \ref{elm14}
 $\phi(W-\ri IW) = -W-\ri IW = -\tau(W-\ri IW)$ whence the condition b) holds.

 Conversely let condition c) hold. Then for  $W\in \mk{w}$ we have $\phi(IW) = - I\phi(W) = IW$ and by Corollary \ref{elm14} we have $I(\mk{w}) \subset \mk{o}$, whence $I$ is admissible.
\epr

\subsection{The classification}
When $2d < rank(\mk{g})$ and $I$ is an admissible $\mk{b}^+$ complex structure we have to determine the action of a matching complex structure $J$ on the subspace $\mk{j}\subset \mk{o}$ (see Definition \ref{mj1}).

\begin{prop} \label{eta1} Let $rank(\mk{u}) = 2d + 2p,\quad p \in 2\NN$. Then $J(\mk{j}^+) = \mk{j}^-$.
 \end{prop}
\bpr
Let $S_1,\dots,S_p$ be a basis of $\mk{j}^+$, then $T_k = \tau(S_k),\ k = 1,\dots,p$ is a basis of $\mk{j}^-$.
Then by  \eqref{mj7} $\{P_\gamma\vert\ \gamma \in \Gamma\}\cup \{S_1,\dots,S_p\}$ is a basis of $\mk{h}^+$ and $\{Q_\gamma\vert\ \gamma \in \Gamma\}\cup\{ T_1,\dots,T_p\}$ is a basis of $\mk{h}^-$.
Slightly changing notation for the elements of the matrix $\ms{J}$ (see \eqref{bo8}, Proposition \ref{lemma about J}) for any $q = 1,\dots,p$ we have
\begin{gather} \label{ee2}
J(S_q)  =  \sum_{\delta\in\Gamma}b_{\delta,q}Q_\delta + \sum_{t=1}^p
b_{t q} T_t + \sum_{\beta \in \Delta^+}\eta_{\beta, q} E_{-\beta}.
\end{gather}
From integrability for $q = 1,\dots,p,\quad \gamma\in\Gamma$ we have
  \begin{gather*}
   0 = N_J(S_q,V_\gamma) = -\left [\sum_{\delta\in\Gamma}b_{\delta,q}Q_\delta + \sum_{t=1}^p
b_{t q} T_t + \sum_{\beta \in \Delta^+}\eta_{\beta, q} E_{-\beta}, E_\gamma\right]\\ + J[S_q,E_\gamma]
 - J\left[ \sum_{\delta\in\Gamma}b_{\delta,q}Q_\delta + \sum_{t=1}^p
b_{t q} T_t + \sum_{\beta \in \Delta^+}\eta_{\beta, q} E_{-\beta},V_\gamma\right] \\
 = - b_{\gamma,q}\gamma(Q_\gamma)E_\gamma - \sum_{\beta \in \Delta^+}
\eta_{\beta, q}[E_{-\beta},E_\gamma]  - \sum_{\beta \in \Delta^+}\eta_{\beta, q} \beta(V_\gamma)JE_{-\beta}.
 \end{gather*}
From Proposition \ref{mxj1} and formula \eqref{ca12} for $\beta \in \Phi_{\gamma}^+$ we have
$$
 JE_{-\beta} = -\tau(JE_\beta) = N_{\gamma,- \beta} \ol{\gamma(V_\gamma)} E_{\gamma-\beta},\quad  \beta(V_\gamma) = -\ri\gamma(V_\gamma) \beta(Q_\gamma)
 $$
therefore we have
\begin{gather}\nonumber
 0  =\ri b_{\gamma,q}E_{\gamma} - \sum_{\beta \in \Phi_\gamma^+}
\eta_{\beta, q}N_{\gamma,-\beta}(1 +\ri \beta(Q_\gamma))E_{\gamma-\beta} \\[-2mm]\label{ee1}  \\[-2mm]
 - \sum_{\beta \in \Delta^+ \setminus  \Phi_\gamma^+}
\eta_{\beta q}([E_{-\beta},E_{\gamma}]  - \beta(V_\gamma) JE_{-\beta}). \nonumber
\end{gather}
From Proposition \ref{mxj1} it follows that
$$\sum_{\beta \in \Delta^+ \setminus  \Phi_\gamma^+}
\eta_{\beta q}([E_{-\beta},E_{\gamma}]  - \beta(V_\gamma) JE_{-\beta}) \in \mk{h} + \sum_{\alpha \in \Delta \setminus \Phi_\gamma^+} \mk{g}(\alpha).
$$
Now from \eqref{ee1} it follows that for any $\beta \in \Phi_\gamma^+$  we have
$$
 \eta_{\beta,q}N_{\gamma,-\beta}(1 +\ri \beta(Q_\gamma)) =  -\ri N_{\gamma,-\beta}\eta_{\beta,q}\alpha(Q_\gamma) =0.
 $$
Where  $\alpha = \mu(\beta)$. But for $\alpha \in \Phi_\gamma^+$ we have $\alpha(Q_\gamma)\neq 0$ (see the end of the proof of Proposition \ref{mxj1}), whence for  $\beta \in \Phi^+ ,\quad q=1,\dots,p$ we have
  \begin{gather} \label{ee3}
   \eta_{\beta,q} = 0.
   \end{gather}
Suppressing all terms containing $\eta_{\beta,q},\ \beta \in \Phi^+$, equation  \eqref{ee1} reduces to

\begin{gather*} 0  = \ri b_{\gamma,q}E_\gamma  -   \sum_{\delta \in \Gamma}
\eta_{\delta,q}([E_{-\delta},E_{\gamma}]  - \delta(V_\gamma) JE_{-\delta}) \\
= \ri b_{\gamma,q}E_\gamma -
\eta_{\gamma, q}(H_\gamma + \ri \gamma(V_\gamma) U_\gamma).
\end{gather*}
By \eqref{ca10} $H_\gamma + \ri \gamma(V_\gamma ) U_\gamma \neq 0$ and we see that for  $\gamma \in \Gamma,\quad q=1,\dots,p$ we have
\begin{gather}\label{ee4}
\eta_{\gamma,q} = b_{\gamma,q}  = 0.
\end{gather}
 Thus we conclude (see \eqref{ee2}, \eqref{ee3},\eqref{ee4}) that for any $q = 1,\dots,p$
\begin{gather}\label{mat1}
J(S_q) = \sum_{k=1}^{p}b_{k q} T_k .
\end{gather}
The proposition is proved.
\epr

 We are ready to present our solution of {\bf Problem B} from the introduction.
\begin{theorem} \label{mai2}
 Let $\mk{u}$ be a compact Lie algebra with $rank(\mk{u}) = 2d + 2p$, where $d$ is the number of roots in the stem $\Gamma$ of $\Delta^+$ and $p$ is a nonnegative even integer.  Let $I$ be an admissible  $\mk{b}^+$-complex structure  on $\mk{u}$.Then any hypercomplex structure extending $I$ may be determined by a complex structure $J$ matching $I$, so that there exists a $p\times p$ complex matrix ${\bf b}$, with $\ol{\bf b}{\bf b} = -1$ and for  $\gamma\in \Gamma,\ \alpha \in \Phi^+_\gamma$ we have
\begin{gather}\label{mat2}
JE_\gamma =  Q_\gamma,\quad JE_\alpha =  \ri N_{\gamma,- \alpha}E_{s_\gamma(\alpha)},\quad
J(S_q) = \sum_{k=1}^p b_{k q}T_k.\,
\end{gather}
 where $S_1,\dots,S_p$ is a basis of $\mk{j}^+$, and $T_k = \tau(S_k),\ k = 1,\dots,p$.
\end{theorem}
\bpr
We choose $\rho_\gamma = 1$ for all $\gamma \in \Gamma$. The first equality is in  \eqref{oo51}. The second is in Proposition \ref{mxj1}. The third follows from Proposition \ref{eta1}.
\epr
Theorems \ref{mai1} and \ref{mai2} allow one to study the parameter spaces for the classes of equivalent hypercomplex structures on a compact connected Lie group ${\bf U}$. This will be done in a subsequent paper, we give only two characteristic examples in the nearest to semisimple case where the equivalence classes of hypercomplex structures are in a bijective correspondence with equivalence classes of admissible complex structures .

\begin{ex}\label{mai5} Let $I$ and $I'$ be two admissible $\mk{b}^+$-complex structures on $su(2d +1)$.
(see Example \ref{fraka}, a)).
Then $I$ is equivalent to $I'$ if and only if either $I = I'$ or  $IH_\gamma = -I'H_\gamma$ for each $\gamma  \in\Gamma$.

The parameter space of equivalence classes of hypercomplex structures on  $SU(2d +1)$ is $\ZZ_2 \backslash {\bf GL}(d,\RR)$.
\end{ex}

 \begin{ex}\label{mai6}
Let ${\bf U} = Sp(d)\times T^d$. The universal covering group is $\wt{\bf U} \cong Sp(d)\times\RR^d$.
Now $\Gamma$ is the set of long roots in $\Delta^+$ (see Example \ref{frakc}). Up to equivalence, there is {\bf exactly one}  left invariant hypercomplex structure on the (noncompact) universal cover group  $\wt{\bf U} $. Indeed all bases $Z_1,\dots,Z_d$ of $\mk{o}_u$ are equivalent under the action of
$GL(d,\RR) \cong \{g\in {\bf Aut}(\bf U)\vert dg(\mk{b}^+) = \mk{b}^+\}$.

The parameter space of equivalence classes of hypercomplex structures on  $Sp(d)\times {\bf T}^d$ is obviously ${\bf GL}_{d}(\ZZ)\backslash {\bf GL}_{d}(\RR)$.
\end{ex}

\begin{ack}
G. K. Dimitrov
acknowledges support from the European Operational program HRD
contract BGO051PO001/07/3.3-02/53 with the Bulgarian
Ministry of Education.
\end{ack}

\end{document}